\documentclass[11pt]{amsart}

\usepackage{amsmath}
\usepackage{amssymb}
\usepackage{bbm}
\usepackage{pdfsync}
\usepackage{esint} 

\newcommand{\C}{{\mathbb C}}       
\newcommand{\R}{{\mathbb R}}       
\newcommand{\N}{{\mathbb N}}

\newcommand{\Z}{{\mathbb Z}}       
\newcommand{\DD}{{\mathcal D}}

\newcommand{\HH}{{\mathcal H}}

\newcommand{\diam}{{\rm diam}}
\newcommand{\dist}{{\rm dist}}

\newcommand{\rf}[1]{{(\ref{#1})}}

\newcommand{\supp}{\operatorname{supp}}

\newcommand{\vphi}{{\varphi}}
\newcommand{\ve}{{\varepsilon}}
\newcommand{\vv}{{\vspace{2mm}}}
\newcommand{\vvv}{{\vspace{3mm}}}
\newcommand{\wt}[1]{{\widetilde{#1}}}

\newcommand{\noi}{\noindent}

\def\Xint#1{\mathchoice
{\XXint\displaystyle\textstyle{#1}}%
{\XXint\textstyle\scriptstyle{#1}}%
{\XXint\scriptstyle\scriptscriptstyle{#1}}%
{\XXint\scriptscriptstyle\scriptscriptstyle{#1}}%
\!\int}
\def\XXint#1#2#3{{\setbox0=\hbox{$#1{#2#3}{\int}$ }
\vcenter{\hbox{$#2#3$ }}\kern-.58\wd0}}

\def\avint{\;\Xint-}

\usepackage{pgf,tikz} 

\definecolor{ffffff}{rgb}{1.0,1.0,1.0}
\definecolor{qqqqff}{rgb}{0.0,0.0,1.0}
\definecolor{ffqqqq}{rgb}{1.0,0.0,0.0}
\definecolor{zzzzqq}{rgb}{0.6,0.6,0.0}
\definecolor{marronet}{rgb}{0.6,0.2,0}
\definecolor{negre}{rgb}{0,0,0}
\definecolor{vermell}{rgb}{0.8,0.05,0.05}
\definecolor{blau}{rgb}{0.3,0.2,1.}
\definecolor{blauclar}{rgb}{0.,0.,1.}
\definecolor{grisfosc}{rgb}{0.25098039215686274,0.25098039215686274,0.25098039215686274}
\definecolor{verd}{rgb}{0.1,0.6,0.1}
\definecolor{taronja}{rgb}{0.9,0.6,0.05}
\definecolor{vermellclar}{rgb}{1.,0.,0.}
\definecolor{verdet}{rgb}{0,0.8,0.1}
\definecolor{blauverd}{rgb}{0,0.4,0.2}
\definecolor{grisclar}{rgb}{0.6274509803921569,0.6274509803921569,0.6274509803921569}

\newtheorem{theorem}{Theorem}[section]
\newtheorem{lemma}[theorem]{Lemma}

\newtheorem*{theorem*}{Theorem}

\theoremstyle{definition}

\newtheorem{example}[theorem]{Example}

\theoremstyle{remark}
\newtheorem{rem}[theorem]{Remark}

\numberwithin{equation}{section}

\newcommand{\brem}{\begin{rem}}
\newcommand{\erem}{\end{rem}}

\begin{document}

\title[Removable singularities for Lipschitz caloric functions]{Removable singularities for Lipschitz caloric functions in time varying domains}

\author{Joan Mateu}
\address{Joan Mateu\\
Departament de Matem\`atiques,  
\\
 Universitat Aut\`onoma de Barcelona and Centre de Recerca Matemàtica
 \\
08193 Bellaterra (Barcelona), Catalonia.
}
\email{mateu@mat.uab.cat}
\author{Laura Prat}
\address{Laura Prat\\
Departament de Matem\`atiques 
\\
Universitat Aut\`onoma de Barcelona and Centre de Recerca Matemàtica
\\
08193 Bellaterra (Barcelona), Catalonia.
}
\email{laurapb@mat.uab.cat
}

\author{Xavier Tolsa}

\address{Xavier Tolsa
\\
ICREA, Passeig Llu\'{\i}s Companys 23 08010 Barcelona, Catalonia\\
 Departament de Matem\`atiques 
\\
Universitat Aut\`onoma de Barcelona and Centre de Recerca Matemàtica
\\
08193 Bellaterra (Barcelona), Catalonia.
}
\email{xtolsa@mat.uab.cat}

\thanks{All authors were supported by 2017-SGR-0395 (Catalonia).  L.P. and X.T. were supported by MTM-2016-77635-P (MINECO, Spain) and J.M. was supported by MTM-2016-75390 (MINECO, Spain).}

\begin{abstract}
In this paper we study  removable singularities for regular $(1,1/2)$-Lipschitz solutions of the heat equation in time varying domains. We introduce an associated Lipschitz caloric capacity and we study its metric and geometric properties and the connection with the $L^2$ boundedness of the singular integral
whose kernel is given by the gradient of the fundamental solution of the heat equation.
\end{abstract}

\maketitle

\section{Introduction}

A compact set $E\subset\C$ is said to be removable for bounded analytic functions if for any open set $\Omega$ containing $E$, every bounded function analytic on $\Omega\setminus E$ has an analytic extension to $\Omega$. In \cite{Ahlfors}, Ahlfors showed that $E$ is removable for bounded analytic functions if and only if $E$ has zero analytic capacity. Analytic capacity is a notion that, in a sense, measures the size of a set as a non removable singularity. 
In the higher dimensional setting, one considers removable sets for Lipschitz harmonic functions: we say that a compact set $E\subset\R^{n+1}$ is removable for Lipschitz harmonic functions if, for each open set $\Omega\subset\R^{n+1}$, every Lipschitz function $f:\Omega\to\R$ that is harmonic in $\Omega\setminus E$ is harmonic in the whole $\Omega$. 
Nowadays,  very complete results are known for removable sets for bounded analytic functions in the plane (see \cite{Tolsa-llibre} for example) and also in the higher dimensional setting for removable sets for Lipschitz harmonic functions (see \cite{volbergsemiad}, \cite{NToV1}, \cite{NToV2}). 
 The Cauchy transform
and the Riesz transforms play a prominent role in their study.

In the present paper we study  removable singularities for regular $(1,1/2)$-Lipschitz solutions of the heat equation in time varying domains.  The parabolic theory in time varying domains is an area that has experienced a lot of activity
in the last years, with fundamental contributions by Hofmann, Lewis,
Murray, Nystr\"om, Silver, and Str\"omqvist \cite{Hofmann-Duke}, \cite{Hofmann-Lewis1}, 
\cite{Hofmann-Lewis2}, \cite{HLN1}, \cite{HLN2}, \cite{Lewis-Murray}, \cite{Lewis-Silver}, \cite{Nystrom-Stromqvist}.

Next we introduce some notation and definitions.
Our ambient space is $\R^{n+1}$ with a generic point denoted as $\bar x=(x,t)\in\R^{n+1}$, where $x\in\R^n$ and $t\in\R$.  We let $\Theta$ denote the heat operator, 
$\Theta=\Delta-\partial_t,$
where $\Delta=\Delta_x$ is the Laplacian with respect to the $x$ variable.
Then, for a smooth function $f$ depending on $(x,t)\in\R^{n+1}$,
$$\Theta(f) = \Delta f -\partial_t f=0$$
is just the heat equation.

Given $\bar x= (x,t)$ and $\bar y = (y,u)$, with $x,y\in\R^n$, $t,u\in\R$,
we consider the parabolic distance in $\R^{n+1}$ defined by
$$\dist_p(\bar x,\bar y) = \max\big(|x-y|, \,|t-u|^{1/2}\big).$$
Sometimes we also write
$|\bar x - \bar y|_p$ instead of $\dist_p(\bar x,\bar y)$.
We denote by $B_p(\bar x,r)$ a parabolic ball (i.e., in the distance $\dist_p$) centered at $\bar x$ with radius $r$. By a parabolic cube $Q$ of side length $\ell$, we mean a set of the form
$$I_1 \times\ldots\times I_n\times I_{n+1},$$
where $I_1,\ldots,I_n$ are intervals in $\R$ with length $\ell$, and $I_{n+1}$ is another interval with
length $\ell^2$. We write $\ell(Q)=\ell$. 

We say that a Borel measure $\mu$ in $\R^{n+1}$ has upper parabolic growth of degree $n+1$ if there
exists some constant $C$ such that
\begin{equation}\label{C=1}
\mu(B_p(\bar x,r))\leq C r^{n+1}\quad\mbox{ for all $\bar x\in\R^{n+1},\, r>0$.}
\end{equation}
Clearly, this is equivalent to saying that any parabolic cube $Q\subset\R^{n+1}$ satisfies $\mu(Q)\leq C'
\ell(Q)^{n+1}$. Given $E\subset\R^{n+1}$, we denote by $\Sigma(E)$ the family of (positive) Borel measures $\mu$ supported on $E$
which have upper parabolic growth of degree $n+1$ with constant $C=1$ in \rf{C=1}.

Throughout the paper $\|\cdot\|_{*,p}$ denotes the norm of the parabolic BMO space:
$$\|f\|_{*,p} = \sup_Q \avint_Q |f -m_Q f|\,dm,$$
where the supremum is taken over all parabolic cubes $Q\subset \R^{n+1}$,
$dm$ stands for the Lebesgue measure in $\R^{n+1}$ and $m_Qf$ is the mean of $f$ with respect to $dm$. 
For a function $f:\R^{n+1}\to\R$, we set 
$$
\partial_t^{1/2}f(x,t)=\int\frac{f(x,s)-f(x,t)}{|s-t|^{3/2}}ds.
$$
We say that a compact set $E\subset \R^{n+1}$ is Lipschitz removable for the heat equation (or Lipschitz caloric removable) if for any open set $\Omega\subset\R^{n+1}$, any function $f:\R^{n+1}\to\R$
such that
\begin{equation}\label{lip11/2}
\|\nabla_x f\|_{L^\infty(\Omega)}<\infty ,\qquad \|\partial_t^{1/2} f\|_{*,\Omega,p} <\infty
\end{equation}
satisfying the heat equation in $\Omega\setminus E$, also satisfies the heat equation in the whole $\Omega$.  
Functions satisfying \rf{lip11/2} are called regular $(1,1/2)$-Lipschitz in the literature (see
\cite{Nystrom-Stromqvist}, for example). So perhaps it would be
more precise to talk about regular $(1,1/2)$-Lipschitz removability. However, we have preferred the simpler terminology of Lipschitz removability for shortness.

Our motivation to study the singularities for regular $(1,1/2)$-Lipschitz functions, with the
parabolic BMO condition in the half derivative with respect to time, comes from the results in 
\cite{Hofmann-Duke}, \cite{Hofmann-Lewis1},  \cite{Lewis-Murray}, and \cite{Lewis-Silver}. In these
works in connections with parabolic singular integrals and caloric layer potential on graphs,
 it has become clear 
 that the right graphs are the ones of functions that are Lipschitz in the space variable and have 
 half time derivative in parabolic BMO. The results that we obtain in this paper (like the ones
 about localization of singularities that we describe below) also confirm that the parabolic BMO condition
 on the half time derivative is a natural assumption.

Given a set $E\subset \R^{n+1}$, we define its Lipschitz caloric capacity by
\begin{equation}\label{def1}
\gamma_\Theta(E) =\sup\{|\langle\nu,1\rangle|:  \nu\in\mathcal D', \,\supp\nu\subset E,\|\nabla_x W*\nu\|_{L^\infty(\R^{n+1})} \leq 1\mbox{ and }\|\partial_t^{1/2} W*\nu\|_{*,p} \leq 1\},
\end{equation}
where  $\mathcal D'$ is the space of distributions in $\R^{n+1}$ and
$W(x,t)$ denotes the fundamental solution of the heat equation in  $\R^{n+1}$, that  is 
$$W(x,t) =\left\{\begin{array}{l} \frac{1}{(4\pi t)^{n/2}}\,e^{-|x|^2/(4t)}\quad \mbox{if }\; t>0\\\\\quad0\hspace{2.65cm}\mbox{if }\;t\leq 0\end{array}\right..$$

We shall now give a brief description of the main results in the paper. 
In Section \ref{section-localization} we deal with a localization result. More concretely, for a distribution $\nu$, we localize the potentials  $\nabla W*\nu$ and $\partial_t^{1/2}W*\nu$ in the $L^\infty-$norm and the parabolic BMO norm respectively. The localization method for the Cauchy potential $\nu*1/z$ in the plane is a basic tool 
 developed by A.G. Vitushkin in the theory of rational approximation in the plane. This was later adapted in \cite{paramonov} for the Riesz potential $\nu*x/|x|^{n}$ in $\R^n$ and used in problems of ${C}^1-$harmonic approximation. These localization results have also been essential to prove the semiadditivity of analytic capacity and of Lipschitz harmonic capacity, see \cite{tolsasemiad} and \cite{volbergsemiad} respectively (see  also \cite{tams} for other related capacities).
 In Section \ref{section-positive} we restrict ourselves to the case when the distribution $\nu$
 in \rf{def1} is a positive measure $\mu$. We show that if $\mu$ has upper parabolic growth of degree $n+1$ and $\nabla_xW*\mu$ is in $L^\infty(\R^{n+1})$, then  $\partial_t^{1/2}W*\mu$ is bounded in the parabolic BMO-norm. This fact will be very useful when studying the capacity $\gamma_{\Theta,+}$, whose definition is 
 analogous to the one in \eqref{def1} but with the supremum restricted to positive measures. 
 
 In Section \ref{section-capacities} we study the connection between Lipschitz caloric removability and the capacity $\gamma_\Theta$.  In particular, we show that
 a compact set $E\subset\R^{n+1}$ is Lipschitz caloric removable if and only if $\gamma_\Theta(E)=0$.
  We also compare the capacity $\gamma_{\Theta}$ to the parabolic Hausdorff content $\mathcal{H}_{\infty,p}^{n+1}$ and we prove that if $E$ has zero $(n+1)-$dimensional parabolic Hausdorff measure, i.e., $\HH^{n+1}_p(E)=0$, then  $\gamma_{\Theta}(E)=0$ too. In the converse direction, we show that if $E$ has parabolic Hausdorff dimension larger than $n+1$, then $\gamma_{\Theta}(E)$ is positive. Hence, the critical parabolic dimension for Lipschitz caloric capacity 
 (and thus for Lipschitz caloric removability)
 occurs in dimension $n+1$, in accordance with the classical case. We remark here that the parabolic Hausdorff measure $\HH^{n+1}_p,$ the parabolic Hausdorff content $\mathcal{H}_{\infty,p}^{n+1},$ and the parabolic Hausdorff dimension are defined as in the Euclidean case 
 (see \cite{Mattila-gmt}, for instance), just replacing the Euclidean distance by the parabolic distance introduced above. Then it turns out
 that $\R^{n+1}$ has parabolic Hausdorff dimension $n+2$.
  
In Section \ref{section-capacities} we also introduce
 a new capacity $\wt\gamma_{\Theta,+}$. We consider the convolution operator $T$ with kernel $K = 
 \nabla_x W$, which is of Calder\'on-Zygmund type in the parabolic space. We denote by $T^*$ its dual operator. Then we set
$\wt\gamma_{\Theta,+}(E)=\sup\mu(E),$
where the supremum is taken over all positive measures measures $\mu\in\Sigma(E)$ such that
\begin{equation*}
\|T\mu\|_{L^\infty(\R^{n+1})} \leq 1,\qquad \|T^*\mu\|_{L^\infty(\R^{n+1})} \leq 1.
\end{equation*}
 We show that the capacity $\wt\gamma_{\Theta,+}$ can be characterized in terms of the $L^2$-norm of $T$
 and that  $\gamma_\Theta\gtrsim\wt\gamma_{\Theta,+}$. 
   Then we show that any subset of positive measure $\HH^{n+1}_p$ of a
regular Lip$(1,1/2)$ graph has positive capacity $\wt\gamma_{\Theta,+}$ and is non-removable. In particular, any subset of positive measure $\HH^{n+1}_p$ of 
a non-horizontal hyperplane (i.e., not parallel to $\R^n\times\{0\}$) is non-removable. Let us remark that any horizontal plane has parabolic Hausdorff dimension $n$, and thus any subset of a horizontal plane is removable.

In the last section of the paper we construct a self-similar Cantor set $E\subset \R^3$ with positive and finite measure $\HH^3_p$ and we show that this is Lipschitz removable for the heat equation. The construction extends easily to $\R^{n+1}$, with $n\geq 1$ arbitrary, but we work in $\R^3$ for simplicity. Our example is inspired by the typical planar $1/4$ Cantor
set in the setting of analytic capacity (see \cite[p.\ 35]{Tolsa-llibre}, for example).

By analogy with what happens with analytic capacity \cite{David-vitus} or Lipschitz harmonic capacity \cite{NToV2} and because of the examples of regular Lip$(1,1/2)$ graphs and the Cantor set mentioned above, 
one
should expect that a set $E\subset \R^{n+1}$ is Lipschitz caloric removable if and only if it is parabolic purely $(n+1)$-unrectifiable in some sense. Remark that it seems natural to define that set $E$ as 
parabolic purely $(n+1)$-unrectifiable if it intersects any regular Lip$(1,1/2)$ graph at most in a 
set of measure $\HH^{n+1}_p$ zero (see \cite{Nystrom-Stromqvist} for some results on parabolic uniform rectifiability).
 A first step in this direction might consist in proving that
$\gamma_\Theta(E)>0$ if and only if $\wt\gamma_{\Theta,+}(E)>0$ (or even that both capacities are comparable). However, there is a big obstacle when trying to follow this approach. Namely, the kernel
$K=\nabla_x W$ is not antisymmetric and thus, if $\nu$ is such that $T\nu = \nabla_x W*\nu$ is in $L^\infty(\R^{n+1})$,
apparently one cannot get any useful information regarding $T^*\nu$. This prevents any direct application 
of the usual $T1$ or $Tb$ theorems from Calder\'on-Zygmund theory, which are  essential tools in the
case of analytic capacity or Lipschitz harmonic capacity. A connected question is the following:
is it true that a set is removable for the heat equation if and only if it is removable for the adjoint heat equation $\Delta f +\partial_t f=0$?

\vv
Some comments about the notation used in the paper: as usual, the letter $C$ stands for an absolute constant which may change its value at different occurrences.
The notation $A\lesssim B$ means that there is a positive absolute constant $C$ such that $A\leq CB$. Also, $A\approx B$ is equivalent to $A\lesssim B\lesssim A$.
\vv

We would like to thank the anonymous referees for the careful reading of the paper and useful suggestions.

\vv

\section{Some  preliminary estimates}\label{one}

In the next lemma we will obtain upper bounds for the kernels $W(x,t), \nabla_xW(x,t), \partial_tW(x,t)$ and $\partial_t^{1/2}W(x,t).$

\begin{lemma}\label{estimates}
For any $\bar x= (x,t), \, x\in \R^n$ and $t\in \mathbb R,$ the following holds:
$$0\leq W(\bar x)\lesssim \frac1{|\bar x|_p^{n}},$$
$$|\nabla_x W(\bar x)|\lesssim \frac1{|\bar x|_p^{n+1}},$$
$$|\partial_t^{1/2} W(\bar x)|\lesssim \frac1{|x|^{n-1}\,|\bar x|_p^2},$$
$$|\partial_t W(\bar x)|\lesssim \frac1{|\bar x|_p^{n+2}}.$$
\end{lemma}

\begin{proof}
To prove the first inequality we use the fact that $e^{-|y|}\lesssim \min(1,|y|^{-n/2})$, and then we get
$$W(\bar x)\lesssim \frac1{t^{n/2}}\,\min\bigg(1,\frac{t^{n/2}}{|x|^n}\bigg) = \frac1{\max(t^{n/2},|x|^n)}
= \frac1{|\bar x|_p^{n}}.$$
Concerning the second estimate in the lemma, we have
$$\nabla_x W(x,t) = c \frac{x}{t^{n/2+1}}\,e^{-|x|^2/(4t)}\,\chi_{\{t>0\}}.$$
So using now that
$e^{-|y|}\lesssim \min(|y|^{-1/2},|y|^{-1-n/2})$, we derive
$$|\nabla_x W(\bar x)| \lesssim c\frac{|x|}{t^{n/2+1}} \,\min\bigg(\frac{t^{1/2}}{|x|},\frac{t^{n/2+1}}{|x|^{n+2}}\bigg) = \frac1{\max(t^{(n+1)/2},|x|^{n+1})}
= \frac1{|\bar x|_p^{n+1}}.$$
For the last inequality, we compute
$$\partial_t W(\bar x) = \bigg(\frac{c_1}{t^{n/2+1}} \,e^{-|x|^2/(4t)} + \frac{c_2|x|^2}{t^{n/2+2}}\,e^{-|x|^2/(4t)}\bigg)\,\chi_{\{t>0\}},$$
and then we argue as above. We leave the details for the reader.

The proof of the third inequality will take some more work. Clearly, we may assume $x\neq0$.
First we write $W(x,t)$ in the form
$$W(x,t) = \frac{c}{|x|^n}\,\bigg(\frac{|x|^2}{4t}\bigg)^{n/2}\,e^{-|x|^2/(4t)} \,\chi_{\{t>0\}}= \frac{c}{|x|^n}\,f\bigg(\frac{4t}{|x|^2}\bigg),$$
where 
$$f(s) = \frac1{s^{n/2}}\,e^{-1/s}\,\chi_{\{s>0\}}.$$
Notice that $f$ is a $C^{\infty}$ function that vanishes at $\infty$.
Then we have
$$\partial_t^{1/2}W(x,t) = \frac{c}{|x|^n}\,\partial_t^{1/2}\bigg[f\bigg(\frac{4\;\cdot}{|x|^2}\bigg)\bigg](t).$$
By a change of variable, it is immediate to check that
$$\partial_t^{1/2}\bigg[f\bigg(\frac{4\;\cdot}{|x|^2}\bigg)\bigg](t)= \frac2{|x|}\,
\partial_t^{1/2} f\bigg(\frac{4t}{|x|^2}\bigg),
$$
and thus
$$\partial_t^{1/2}W(x,t) = \frac{c}{|x|^{n+1}}\,\partial_t^{1/2}f\bigg(\frac{4t}{|x|^2}\bigg).$$
We will show below that, for any $t\in\R$,
\begin{equation}\label{eqdt1}
|\partial_t^{1/2}f(t)| \lesssim \min(1,|t|^{-1}).
\end{equation}
Clearly, this implies that
$$|\partial_t^{1/2}W(x,t)|\lesssim \frac{1}{|x|^{n+1}}\,\min\bigg(1,\frac{|x|^2}{t}\bigg)
= \frac1{\max(|x|^{n+1}, |x|^{n-1}t)} = \frac1{|x|^{n-1}\,|\bar x|_p^2}.$$

The proof of \rf{eqdt1} is a straightforward but lengthy calculation. We split the integral as follows:
\begin{align*}
|\partial_t^{1/2}f(t)| &\leq \int_{|s|\leq |t|/2}\!\! \frac{|f(s) - f(t)|}{|s-t|^{3/2}}\,ds 
+ \int_{|t|/2<|s|\leq 2|t|} \!\! \frac{|f(s) - f(t)|}{|s-t|^{3/2}}\,ds
+ \int_{|s|> 2|t|} \!\! \frac{|f(s) - f(t)|}{|s-t|^{3/2}}\,ds\\
& = I_1 + I_2 + I_3.
\end{align*}
To estimate $I_1$ we use that $|s-t|\approx|t|$ in its domain of integration, and then we get
\begin{equation}\label{eqI1*}
I_1\lesssim \frac1{|t|^{3/2}}\int_{|s|\leq|t|/2} \frac1{|s|^{n/2}}\,e^{-1/|s|}\,ds
+ \frac1{|t|^{3/2}}\int_{|s|\leq|t|/2} \frac1{|t|^{n/2}}\,e^{-1/|t|}\,ds.
\end{equation}
The second summand equals
$$\frac C{|t|^{3/2}} \frac1{|t|^{n/2-1}}\,e^{-1/|t|} =  \frac C{|t|^{(n+1)/2}}\,e^{-1/|t|}.$$
 The integral in the first summand  of \rf{eqI1*} can be estimated as follows:
\begin{align*}
\int_{|s|\leq|t|/2} \frac1{|s|^{n/2}}\,e^{-1/|s|}\,ds & \leq e^{-1/(2|t|)}\int_{|s|\leq 1} \frac1{|s|^{n/2}}\,e^{-1/(2|s|)}\,ds + e^{-1/(2|t|)} \int_{1\leq |s|\leq|t|/2} \frac1{|s|^{n/2}}\,ds\\
& \lesssim e^{-1/(2|t|)} (1+ |t|^{1/2}).
\end{align*}
Hence,
$$I_1\lesssim \frac1{|t|^{3/2}}e^{-1/(2|t|)} (1+ |t|^{1/2}) +
 \frac1{|t|^{(n+1)/2}}\,e^{-1/|t|} \lesssim \min\bigg(1,\frac1{|t|}\bigg).$$

To deal with $I_2$, we distinguish two cases, according to wheter $s$ has the same sign as $t$ or not. In the first case we write $s\in Y$, and in the second one, $s\in N$. In the case $s\in N$, with 
$|t|/2\leq|s|\leq 2|t|$, it turns out that $|s-t|\approx |t|$, and thus
\begin{align*}
I_{2,N} & :=\int_{s\in N,|t|/2\leq|s|\leq 2|t|}  \frac{|f(s) - f(t)|}{|s-t|^{3/2}}\,ds\\
&\lesssim \frac1{|t|^{3/2}}\int_{|s|\leq 2|t|} \frac1{|s|^{n/2}}\,e^{-1/|s|}\,ds
+ \frac1{|t|^{3/2}}\int_{|s|\leq2|t|} \frac1{|t|^{n/2}}\,e^{-1/|t|}\,ds.
\end{align*}
Observe that this last expression is very similar to the right hand side of \rf{eqI1*}.
Then, by almost the same arguments we deduce that
$$I_{2,N}\lesssim \min\bigg(1,\frac1{|t|}\bigg).$$
To deal with the case when the sign of $s$ is the same as the one of $t$ (i.e., $s\in Y$),
we take into account that
$$|f(s)-f(t)|\leq \sup_{\xi\in[s,t]}|f'(\xi)|\,|s-t|,$$
Since in this case $|t|/2\leq|\xi|\leq 2|t|$, it is immediate to check that for this $\xi$ we  have
$$|f'(\xi)|\lesssim \frac1{|t|^{n/2+1}}\,e^{-1/(4|t|)}.$$
Thus,
\begin{align*}
I_{2,Y} & :=\int_{s\in Y,|t|/2\leq|s|\leq 2|t|}  \frac{|f(s) - f(t)|}{|s-t|^{3/2}}\,ds\lesssim \frac1{|t|^{n/2+1}}\,e^{-1/(4|t|)} \int_{|t|/2\leq|s|\leq 2|t|} \frac{|s-t|}{|s-t|^{3/2}}\,ds\\
&\leq \frac1{|t|^{n/2+1}}\,e^{-1/(4|t|)} \int_{|s|\leq 2|t|} \frac1{|s-t|^{1/2}}\,ds\lesssim
\frac1{|t|^{(n+1)/2}}\,e^{-1/(4|t|)}\lesssim \min\bigg(1,\frac1{|t|}\bigg).
\end{align*}

Finally, concerning $I_3$, taking into account that $|s-t|\approx|s|\gtrsim|t|$ in the domain of integration,
$$I_3\lesssim \int_{|s|> 2|t|}   \frac{e^{-1/|s|} + e^{-1/|t|}}{|t|^{n/2}|s|^{3/2}} \,ds\lesssim
\int_{|s|> 2|t|}   \frac{e^{-1/|s|}}{|t|^{n/2}|s|^{3/2}} \,ds +  \frac{e^{-1/|t|}}{|t|^{n/2}} \int_{|s|> 2|t|}   \frac1{|s|^{3/2}} \,ds.$$
It is immediate to check that none of the two summands exceeds $C\min(1,|t|^{-1})$.
So gathering all the estimates above, the claim \rf{eqdt1} follows.
\end{proof}


\vv

\section{Localization}\label{section-localization}

Let $\vphi:\R^{n+1}\to\R$ be a $C^2$ function. We say that $\vphi$ is admissible for a parabolic cube $Q$ 
if it is supported on $Q$ and satisfies
\begin{equation}\label{admis}
\|\nabla_x \vphi\|_\infty \leq \frac1{\ell(Q)},\qquad \|\Delta\vphi\|_\infty + \|\partial_t\vphi\|_\infty\leq \frac1{\ell(Q)^2}.
\end{equation}
The main objective of this section is to show the following localization result.

\begin{theorem}\label{teoloc}
Let 
 $\nu$ be a distribution in $\R^{n+1}$ such that 
$$\|\nabla_x W*\nu\|_\infty \leq 1,\qquad \|\partial_t^{1/2} W*\nu\|_{*,p} \leq 1.$$
Let $\vphi$ be a $C^2$ function admissible for a parabolic cube $Q\subset\R^{n+1}$. Then
$$\|\nabla_x W*(\vphi\nu)\|_\infty \lesssim 1,\qquad \|\partial_t^{1/2} W*(\vphi\nu)\|_{*,p} \lesssim 1.
$$
\end{theorem}

\vv

We say that a distribution $\nu$ in $\R^{n+1}$ has upper parabolic growth of degree $n+1$ if there
exists some constant $C$ such that, given any parabolic cube $Q$ and any function
 $C^2$ function $\vphi$ admissible for $Q$, it holds
 $$|\langle \nu,\vphi\rangle|\leq C\ell(Q)^{n+1}.$$
It is immediate to check that this definition is coherent with the one in \rf{C=1} for 
 positive measures. If we want to be precise about the precise constant involved in the definition, we will say that $\nu$ has upper parabolic $C$-growth of degree $n+1$.

\vv
Before proving Theorem \ref{teoloc} we need several lemmas.  The first one shows that every distribution $\nu$ satisfying the hypotheses of Theorem \ref{teoloc} has upper parabolic growth of degree $n+1$.

\begin{lemma}\label{lemgrow}
Let $\nu$ be a distribution in $\R^{n+1}$ such that 
$$\|\nabla_x W*\nu\|_\infty \leq 1,\qquad \|\partial_t^{1/2} W*\nu\|_{*,p} \leq 1.$$
Then $\nu$ has upper parabolic $C$-growth of degree $n+1$, where $C$ is some absolute constant.
\end{lemma}

\begin{proof}
Let $\vphi$ be a $C^2$ function admissible for a parabolic cube $Q$. Since $W$ is the fundamental solution of $\Theta$, we can write
$$|\langle \nu,\vphi \rangle|=|\langle \nu,\Theta\vphi*W \rangle|\leq|\langle W*\nu,\Delta\vphi \rangle|+|\langle W*\nu,\partial_t\vphi \rangle|=I_1+I_2.$$
To estimate $I_1$ we use that $\|\nabla_x \vphi\|_\infty \leq \frac1{\ell(Q)}$ and $\|\nabla_x W*\nu\|_\infty \leq 1$:
$$I_1=|\langle \nabla_x W*\nu,\nabla_x\vphi \rangle|\leq\|\nabla_x W*\nu\|_\infty\int|\nabla_x\vphi|\,dm\leq\ell(Q)^{n+1}.$$

For $I_2$ we consider the function $g=\partial_t\vphi*_tk$, with $k(t)=|t|^{-1/2}$ and $*_t$ being the convolution on the $t$ variable. Taking the Fourier transform on the variable $t$, we get 
$\partial_t \vphi = c\,\partial_t^{1/2} g$, for a suitable absolute constant $c\neq0$. 
Write $Q=Q_1\times I_Q$, with $Q_1\subset\R^n$ being a cube of side length $\ell(Q)$ and $I_Q\subset\R$ an interval of length $\ell(Q)^2$. 
Because of the
zero mean of $\partial_t \vphi$ (integrating with respect to $t$), it is easy to check that 
$|g(x,t)|$ decays at most like $|t|^{3/2}$ at infinity. Indeed, for $t\notin 2 I_Q$, denoting
by $s_Q$ the center of $I_Q$,
\begin{align}\label{eq3232}
|g(x,t)|& =\Big|\int_{I_Q}\frac{\partial_s\vphi(x,s)}{|t-s|^{1/2}}ds\Big| = \Big|\int_{I_Q}\partial_s\vphi(x,s)\Big(\frac{1}{|t-s|^{1/2}} - \frac{1}{|t-s_Q|^{1/2}}\Big)ds\Big|\\
& \lesssim \frac{\ell(I_Q)}{|t-s_Q|^{3/2}}\int_{I_Q}|\partial_s\vphi(x,s)|ds \lesssim \frac{\ell(I_Q)}{|t-s_Q|^{3/2}}.\nonumber
\end{align}
Together with the fact that $\supp g\subset Q_1\times \R$, this implies that $g\in L^1(\R^{n+1})$.
Further, it is easy to check that $\int g\,dm=0$, for example with the help of the Fourier transform in $t$.

Using the zero average property of $g$, writing $f=\partial_t^{1/2}W*\nu$, we have
\begin{align*}
I_2&= |\langle W*\nu,c\,\partial_t^{1/2} g \rangle| =
|\langle f,c\,g \rangle| = \bigg|c\int \left(f-m_Qf\right)\,g\,dm\bigg|\\
&\lesssim 
\int_{2Q}\left|f-m_Qf\right|\,|g|\,dx\;dt+\int_{\R^{n+1}\setminus 2Q}\left|f-m_Qf\right|\,|g|\,dm\\
&=I_{21}+I_{22}.
\end{align*}
Since for $t\in 4I_Q$, 
$$|g(x,t)|\lesssim \int_{I_Q}\frac{|\partial_t\vphi(x,s)|}{|t-s|^{1/2}}ds\lesssim \|\partial_t\vphi\|_\infty\ell(I_Q)^{1/2}\lesssim \frac1{\ell(Q)},$$
we have $I_{21}\lesssim\|f\|_{*,p}\,\ell(Q)^{n+2}\ell(Q)^{-1}\leq\ell(Q)^{n+1}$.

For $I_{22}$, we split the domain of integration in annuli. Write $A_i= 2^{i} Q\setminus 2^{i-1} Q$ for $i\geq1$.
Remark that for a parabolic cube $Q = Q_1\times I_Q$, we denote
$$2^i Q = 2^i Q_1 \times 2^{2i} I_Q,$$
so that $2^iQ$ is a parabolic cube too (notice that if $Q$ is centered at the origin and we consider the parabolic dilation $\delta_\lambda(x,t) = (\lambda x,\lambda^{2}t)$, $\lambda>0$, we have $2^iQ = \delta_{2^i}(Q)$).
Then, using the decay of $g$ given by \rf{eq3232}, we get
\begin{equation}\label{eq3434}
I_{22}\lesssim\sum_{i=1}^\infty \frac{\ell(Q)^2}{\ell(2^iQ)^{3}}\left(\int_{A_i\cap \supp g}|f-m_{2^{i}Q}f|\,dm+ \int_{A_i\cap \supp g}|m_{2^{i}Q}f-m_{Q}f|\,dm\right).
\end{equation}
To estimate the first integral on the right hand side, recall that 
$\supp g\subset Q_1\times \R$. Using H\"older's inequality with some 
exponent $q\in(0,\infty)$ to be chosen in a moment and the fact that
$f\in BMO_p$ (together with John-Nirenberg), then we get:
\begin{align*}
\int_{A_i\cap \supp g}|f-m_{2^{i}Q}f|\,dm &
\leq \left(\int_{2^iQ}|f-m_{2^{i}Q}f|^q\,dm\right)^{1/q} \,m(\supp g\cap 2^iQ)^{1/q'}\\
& \lesssim \ell(2^iQ)^{(n+2)/q}\,(\ell(Q)^n\,\ell(2^iQ)^2)^{1/q'} = \ell(2^iQ)^{(n/q) + 2}\,\ell(Q)^{n/q'}.
\end{align*}
For the last integral on the right hand side of \rf{eq3434}, we write
$$\int_{A_i\cap\supp g}|m_{2^{i}Q}f-m_{Q}f|\,dm\lesssim i\,m(2^iQ\cap \supp g) \leq i\,\ell(Q)^n\,\ell(2^iQ)^2.$$
Therefore,
$$I_{22}\lesssim\sum_{i=1}^\infty \frac{\ell(Q)^2}{\ell(2^iQ)^{3}}\,\Big(
\ell(2^iQ)^{(n/q)+ 2}\,\ell(Q)^{n/q'} + i\,\ell(Q)^n\,\ell(2^iQ)^2\Big).$$
Choosing $q>n$, we get
$$I_{22}\lesssim \ell(Q)^{n+1}.$$
\end{proof}

\vv
Before going to the next lemma, recall that a function $f(x,t)$ defined in $\mathbb R^{n+1}$ is Lip$(1/2)$ (or H\"older $1/2$) in the $t$ variable if
$$\|f\|_{Lip_{1/2,t}} = \sup_{x\in\R^n,\,t,u\in\R} \frac{|f(x,t)-f(x,u)|}{|t-u|^{1/2}}<\infty.$$
It is known that functions $f$ with
$\nabla_x f\in L^\infty(\R^{n+1})$ and $\partial_t^{1/2}f\in BMO_p(\mathbb R^{n+1})$ are Lip$(1/2)$ in $t$. More precisely,
$$\|f\|_{Lip_{1/2,t}}\lesssim \|\nabla_x f\|_{L^\infty(\R^{n+1})} + \|\partial_t^{1/2}f\|_{*,p}.$$
See \cite[Lemma 1]{Hofmann} and \cite[Theorem 7.4]{Hofmann-Lewis1}.

\vv

\begin{lemma}\label{lemlocnabla}
Let $\nu$ be a distribution in $\R^{n+1}$ such that 
$$\|\nabla_x W*\nu\|_\infty \leq 1,\qquad \| W*\nu\|_{Lip_{1/2,t}} \leq 1.$$
Then, if $\vphi$ is a $C^2$ function admissible for some parabolic cube $Q\subset\R^{n+1}$, we have
$$\|\nabla_x W*(\vphi\nu)\|_\infty \lesssim 1.$$
\end{lemma}

\begin{proof}
Notice that for $f$ and $g$ in $C^2$ we have
$\Theta (fg)=g\Theta f+f\Theta g+2\nabla_x f\nabla_x g.$ Therefore, since $W$ is the fundamental solution of $\Theta$, for any constant $c$ we can write
\begin{align}\label{eqhfk311}
\Theta(\vphi\,(W*\nu-c)) &= \vphi\,\Theta(W*\nu - c) + \Theta\vphi\,(W*\nu - c) + 2\,\nabla_x\vphi\cdot
(\nabla_xW*\nu)\\
& = \vphi\,\nu + \Theta\vphi\,(W*\nu - c) + 2\,\nabla_x\vphi\cdot
(\nabla_xW*\nu). \nonumber
\end{align}
Therefore,
\begin{equation}\label{3terms}
\nabla_xW*(\vphi\nu)=\nabla_x(\vphi(W*\nu-c))-\nabla_xW*\left(\Theta\vphi(W*\nu-c)\right)-2\nabla_xW*(\nabla_x\vphi(\nabla_xW*\nu)).
\end{equation}
To estimate the $L^\infty$ norm of \eqref{3terms}, write $Q=Q_1\times I_Q$, where $Q_1\subset\R^n$ is a cube of side length $\ell(Q)$ and $I_Q\subset\R$ an interval of length $\ell(Q)^2$ and choose $c=W*\nu(x_Q,t_Q)$, with $(x_Q, t_Q)$ being the center of the parabolic cube $Q$. Since $W*\nu$ is a Lipschitz function on the $x$ variable and Lip$(1/2)$ on the $t$ variable, for $\bar x=(x,t)\in Q$ we can write
\begin{equation}\label{lip}
\begin{split}
|W*\nu(x,t)-W*\nu(x_Q,t_Q)|&\leq |W*\nu(x,t)-W*\nu(x_Q,t)|\\&\quad+|W*\nu(x_Q,t)-W*\nu(x_Q,t_Q)|\\&\lesssim \ell(Q)+(\ell(Q)^2)^{1/2}\lesssim \ell(Q).
\end{split}
\end{equation}
Using this estimate together with  $\|\nabla_x W*\nu\|_\infty\leq 1$ and the fact that $\vphi$ is admissible for $Q$, we get
$$\|\nabla_x(\vphi(W*\nu-c))\|_\infty\leq\|\nabla_x\vphi\|_\infty\|W*\nu-W*\nu(x_Q,t_Q)\|_\infty+\|\vphi\|_\infty\|\nabla_xW*\nu\|_\infty\lesssim 1.$$

We claim now that if $g$ is a function suppported on $Q$ and such that $\|g\|_\infty\leq\ell(Q)^{-1}$, then $\|\nabla_xW*g\|_\infty\lesssim 1$.
Once the claim is proved, to estimate the $L^\infty-$norm of the second and third terms in \eqref{3terms}, take  $g=\Theta\vphi(W*\nu-c)$ (recall that we have chosen $c=W*\nu(x_Q,t_Q)$) 
and  $g=\nabla_x\vphi(\nabla_xW*\nu)$  respectively. Notice that in the first case the bound $\|g\|_\infty\leq\ell(Q)^{-1}$ is obtained by using \eqref{lip} and the fact that $\vphi$ is admissible for $Q$, 
while in the second case, one uses the admissibility of $\vphi$ together with $\|\nabla_xW*\nu\|_\infty\leq 1$. So the claim applies to both terms, and we therefore obtain $\|\nabla_xW*(\vphi\nu)\|_\infty\lesssim 1.$  

To prove the claim, notice that for $\bar y=(y,s)$,
$$\frac{1}{|\bar y|^{n+1}_p} =\frac1{(\max(|y|,s^{1/2}))^{n+1}}\leq\frac1{|y|^{n-1/2}}\frac1{s^{3/4}}.$$
Take a function $g$ supported on $Q$ and such that $\|g\|_\infty\leq\ell(Q)^{-1}$. For $\bar x\in2 Q$, using Lemma \ref{estimates} we have
\begin{align*}
|\nabla_xW*g(\bar x)|&\leq\|g\|_\infty\int_Q\frac{dm(\bar y)}{|\bar x-\bar y|_p^{n+1}}\\&\leq\frac1{\ell(Q)}\int_{Q_1}\frac{dy}{|x-y|^{n-1/2}}\int_{I_Q}\frac{ds}{|t-s|^{3/4}}\\&\lesssim\frac{\ell(Q)^{1/2}(\ell(Q)^2)^{1/4}}{\ell(Q)}=1.
\end{align*}
and if $\bar x\in (2Q)^c$, then $|\bar x-\bar y|_p^{n+1}\geq\ell(Q)^{n+1}$. Therefore
\begin{align*}
|\nabla_xW*g(\bar x)|&\leq\|g\|_\infty\int_Q\frac{dm(\bar y)}{|\bar x-\bar y|_p^{n+1}}\lesssim\frac{\ell(Q)^{n+2}}{\ell(Q)\ell(Q)^{n+1}}=1.
\end{align*}
Hence $\|\nabla_xW*g\|_\infty\lesssim 1$. This finishes the proof of the claim and the lemma.
\end{proof}

\vv
\begin{lemma}\label{lem3.3}
Let $\nu$ be a distribution in $\R^{n+1}$ such that 
$$\|\nabla_x W*\nu\|_\infty \leq 1,\qquad \| W*\nu\|_{Lip_{1/2,t}} \leq 1.$$
Then, if $\vphi$ is a $C^2$ function admissible for some parabolic cube $Q\subset\R^{n+1}$, we have
$$\| W*(\vphi\nu)\|_{Lip_{1/2,t}}  \lesssim 1.$$
\end{lemma}

\begin{proof}
For any constant $c$, from the identity \rf{eqhfk311} we can write  
\begin{equation}\label{eqparts}
W*\varphi\nu=\varphi(W*\nu-c)-2W*(\nabla_x\varphi\nabla_x (W*\nu))
-W*((W*\nu-c)\Theta\varphi).
\end{equation}
Set $\bar{x}=(x,t)$ and $\tilde{x}=(x,r),$ with $x\in \mathbb R^n$ and $t,r\in\R$. Then 
\begin{align*}
W*(\varphi\nu)(\bar{x})-W*(\varphi\nu)(\tilde{x})&=
\big(\varphi(\bar{x})(W*\nu(\bar{x})-c)-\varphi(\tilde{x})(W*\nu(\tilde{x})-c)\big)\\
&\quad+ \big(-2W*(\nabla_x\varphi \,\nabla_x (W*\nu))(\bar{x})+2W*(\nabla_x\varphi \,\nabla_x (W*\nu))(\tilde{x})\big)\\
&\quad+\big(-W*((W*\nu-c)\Theta\varphi)(\bar{x})+W*((W*\nu-c)\,\Theta\varphi)(\tilde{x})\big)\\&
=A+B+C. 
\end{align*}
We start with the term $A$. If $\bar{x}, \tilde{x}\notin Q$, then $A=0$. Otherwise, let us assume that $\bar{x}\in Q$ and take 
$c=W*\nu(\bar x_{Q})$, where $\bar x_Q$ is the center of $Q$.
Choose a point $\tilde x'$ such that
$\tilde x'= \tilde x$  when $\tilde{x}\in Q$, and otherwise take $\tilde x' \in 2Q$ of the form $\tilde x' =(x,r')$ 
satisfying $ |\tilde x'-\bar{x}|\le |\tilde{x}-\bar{x}|$. Observe that in any case we have
$$\varphi(\tilde{x})(W*\nu(\tilde{x})-c) = \varphi(\tilde{x}')(W*\nu(\tilde{x}')-c)\big)$$
and 
$$|\bar x - \tilde x'|_p\leq \min\big(C\ell(Q),\,|\bar x - \tilde x|_p\big).$$
Then we have 
\begin{align*}
|A|&\le|\varphi(\tilde x')-\varphi(\bar{x})|\,|W*\nu(\tilde x')-c|
+|\varphi(\bar{x})|\,
|W*\nu(\tilde x')-W*\nu(\bar{x})|\\
&\lesssim\frac{|r'- t|}{\ell(Q)^2}\,\ell(Q)+|r'-t|^{1/2}\lesssim |r'-t|^{1/2}\leq |r-t|^{1/2}.
\end{align*}

 \vv
 To estimate the terms $B$ and $C$ we need the following result.
 \begin{lemma}\label{bounded}
Let $g$ be a function supported on a parabolic cube $Q$ and such that $\|g\|_{\infty}\lesssim\frac{1}{\ell(Q)}$.
 Then $\|W*g\|_{Lip_{1/2,t}}\lesssim 1$.
\end{lemma}	

Using Lemma \ref{bounded} we can finish the proof of Lemma \ref{lem3.3}. To estimate $B$ choose $g=\nabla_x\varphi(\nabla_x W*\nu)$. Then clearly $\|g\|_{\infty}\lesssim\frac{1}{\ell(Q)}$
and thus $|B|\lesssim |t-r|^{1/2}$.
For the term $C$, set $g =(W*\nu-c)\Theta\varphi$, with $c=W*\nu(\bar x_{Q})$, $\bar x_Q=(x_Q,t_Q)$ being the center of $Q$. Then, 
for all $\bar y= (y,s)\in Q$,  
\begin{align*}
|W*\nu(\bar y)-W*\nu(\bar x_Q)|&\le 
|W*\nu(y,s)-W*\nu(x_Q,s)|+|W*\nu(x_Q,s)-W*\nu(x_Q,t_Q)|\\
&\le\ell(Q)\|\nabla_x W*\nu\|_{\infty}+|s-t_Q|^{1/2}\|W*\nu\|_{Lip_{1/2,t}}\lesssim\ell(Q).\\
\end{align*}
Consequently $\|g\|_{\infty}\lesssim\ell(Q)\|\Theta\varphi\,\|_{\infty}\lesssim\ell(Q)^{-1}$. 
\end{proof}
\vv

\begin{proof}[Proof of Lemma \ref{bounded}]
Set $\bar{x}=(x,t)$ and $\tilde{x}=(x,r),$ where $x\in \mathbb R^n$ and $t,r\in \R$. Then 
\begin{align*}
W*g(x,t)&= C_n \iint \frac{1}{(t-u)^{n/2}}e^{-\frac{|x-z|^2}{4(t-u)}}g(z,u)\chi_{\{u<t\}}\,dz\,du= \\
&=C_n\iint\frac{1}{(t-u)^{1/2}}\frac{1}{|x-z|^{n-1}}f\Big(\frac{|x-z|^2}{t-u}\Big)g(z,u)\,dz\,du,
\end{align*}
where $f(s)=s^{(n-1)/2}e^{-s}\chi_{\{s>0\}}$.

So, taking $Q=Q_1\times I_Q,$ with $Q_1\subset \mathbb R^n$ being a cube of side length $\ell(Q)$ and $I_Q\subset \mathbb R$ an interval of length $\ell(Q)^2$, we have 
\begin{align*}
|W*g(\bar{x}&)- W*g(\tilde{x})|\\
&\lesssim \frac{1}{\ell(Q)}\iint_Q\frac{1}{|x-z|^{n-1}}\Big|\frac{1}{(t-u)^{1/2}}f\Big(\frac{|x-z|^2}{t-u}\Big)
-\frac{1}{(r-u)^{1/2}}f\Big(\frac{|x-z|^2}{r-u}\Big)\Big|\,dz\,du\\
&\lesssim 
\frac{1}{\ell(Q)} \int_{Q_1}\frac{dz}{|x-z|^{n-1}}\int_{I_Q}
\Big|\frac{1}{|t-u|^{1/2}}-
\frac{1}{|r-u|^{1/2}}\Big|\,du \\
&\quad+\frac{1}{\ell(Q)}\int_{Q_1}\frac{dz}{|x-z|^{n-1}}
\int_{I_Q}\frac1{|t-u|^{1/2}}\Big|f\Big(\frac{|x-z|^2}{t-u}\Big)-f\Big(\frac{|x-z|^2}{r-u}\Big)\Big|\,du\\
&=A+B, 
\end{align*} 
where in the last inequality we have used that $\|f\|_{\infty}\lesssim 1.$	
Now, 
\begin{align*}
A&\lesssim
\int_{|t-u|\leq2|t-r|}\Big|\frac{1}{|t-u|^{1/2}}-\frac{1}{|r-u|^{1/2}}\Big|\,du+ \int_{|t-u|>2|t-r|}\Big|\frac{1}{|t-u|^{1/2}}-\frac{1}{|r-u|^{1/2}}\Big|\,du\\
&=A_1+A_2,
\end{align*}
with
\begin{equation*}
A_1\le \int_{|t-u|\leq 2|t-r|}\frac{du}{|t-u|^{1/2}}+\int_{|r-u|<3|t-r|}\frac{du}{|r-u|^{1/2}}
\lesssim |t-r|^{1/2}
\end{equation*}
and
\begin{equation*}
A_2\lesssim\int_{|t-u|>2|t-r|}\frac{|t-r|}{|t-u|^{3/2}}du\lesssim  |t-r|^{1/2}.
\end{equation*}
Finally, to estimate $B$ we will use that for $|t-u|>2|t-r|$,
\begin{align*}
\Big|f\Big(\frac{|x-z|^2}{t-u}\Big)-f\Big(\frac{|x-z|^2}{r-u}\Big)\Big|&\le \|f'\|_{\infty, I} |x-z|^2\Big|\frac{1}{t-u}-\frac{1}{r-u}\Big|\le |x-z|^2\frac{|t-r|}{|t-u|^2}\|f'\|_{\infty, I} , 
\end{align*}
where $I$ is the interval $\big[\frac{|x-z|^2}{t-u},\frac{|x-z|^2}{r-u}\big].$ 
In the case $|t-u|\leq2|t-r|$, we just take into account that $\|f\|_\infty\leq1$.
Then,
\begin{align*}
B&\le\frac{1}{\ell(Q)}\int_{|t-u|\le2|t-r|}\int_{z\in Q_1} \frac{1}{|x-z|^{n-1}}\frac{1}{|t-u|^{1/2}}\,du\,dz\\
&\quad+\frac{1}{\ell(Q)}\int_{|t-u|>2|t-r|}\int_{z\in Q_1} \frac{1}{|x-z|^{n-1}}\frac{1}{|t-u|^{1/2}}\frac{|x-z|^2|t-r|}{|t-u|^2}\|f'\|_{\infty, I}\,du\,dz\\
&=B_1+B_2.
\end{align*}

The term $B_1$ is clearly bounded by $C|t-r|^{1/2}$. To estimate $B_2$ we use that $f$ is a smooth function satisfying $|f'(s)|\lesssim |s|^{-1}$  and that $|s|\approx \frac{|x-z|^2}{|t-u|}$
for all $s\in I$. Therefore,
\begin{align*}
&B_2\lesssim\frac{1}{\ell(Q)}
\int_{|t-u|>2|t-r|}\int_{z\in Q_1} 
\frac{|t-r|}{|x-z|^{n-3}|t-u|^{5/2}}\frac{|t-u|}{|x-z|^2}\,du\,dz\lesssim|t-r|^{1/2}.\\
\end{align*}
\end{proof}

\vv
\begin{lemma}\label{keylemma}
Let $\nu$ be a distribution in $\R^{n+1}$ such that 
$$\|\nabla_x W*\nu\|_\infty \leq 1,\qquad \|\partial_t^{1/2} W*\nu\|_{*,p} \leq 1.$$
Let $Q,R\subset\R^{n+1}$ be parabolic cubes such that $Q\subset R$.
If $\vphi$ is a $C^2$ function admissible for  $Q$, then we have
$$\int_R |\partial_t^{1/2} W*(\vphi\nu)|\,dm  \lesssim \ell(R)^{n+2}.$$
\end{lemma}

\begin{proof}
From the integration by parts formula \rf{eqparts}, we infer that
$$\partial_t^{1/2} W*(\vphi\nu)= \partial_t^{1/2}\big[\varphi(W*\nu-c)\big]-
2\partial_t^{1/2}W*(\nabla_x\varphi\,\nabla_x W*\nu)
-\partial_t^{1/2}W*((W*\nu-c)\,\Theta\varphi).
$$
We choose $c=W*\nu(\bar x_Q)$, where $\bar x_Q$ is the centre of $Q$.

First we will estimate the $L^1$ norm on $R$ of the last two terms.
We denote $g_1 = \nabla_x\varphi\,\nabla_x W*\nu$ and $g_2 = (W*\nu-c)\,\Theta\varphi$. Notice that,   
$\supp g_i\subset Q$ for $i=1,2$, and also
$$\|g_1\|_\infty\lesssim \frac1{\ell(Q)},$$
Also, from the respective Lip and Lip $1/2$ conditions on the $x$ and $t$ variables,
it follows that
$$|W*\nu(\bar x)-W*\nu(\bar x_Q)|\lesssim \ell(Q)\quad\mbox{ for all $\bar x\in Q$.}$$
Therefore,
$$\|g_2\|_\infty \leq \|W*\nu-W*\nu(\bar x_Q)\|_{\infty,Q}\,\|\Theta\vphi\|_\infty\lesssim \ell(Q)\,\frac1{\ell(Q)^2} = \frac1{\ell(Q)}.$$
Next notice that, by Lemma \ref{estimates},
$$|\partial_t^{1/2} W(\bar x)|\lesssim \frac1{|x|^{n-1}\,|\bar x|_p^2}\leq
\frac1{|x|^{n-1}\, |x|^{1/2}\,|t|^{3/4}} = \frac1{|x|^{n-1/2}\, |t|^{3/4}}.$$
Then, writing
 $Q=Q_1\times I_Q$, where $Q_1$ is a cube with side length $\ell(Q)$ in $\R^n$ and $I_Q$ is an interval of length $\ell(Q)^2$,
we deduce that, for any $\bar x\in\R^{n+1}$,
\begin{align*}
|\partial_t^{1/2} W*g_i(\bar x)| & \lesssim \frac1{\ell(Q)} \int_Q |\partial_t^{1/2} W(\bar x - \bar y)|\,d\bar y\\
& \lesssim \frac1{\ell(Q)}\int_{x\in Q_1}\frac1{|x-y|^{n-1/2}}\,dy \int_{u\in I_Q} \frac1{|t-u|^{3/4}}\,du\\
& \lesssim \frac1{\ell(Q)}\,\ell(Q)^{1/2}\,(\ell(Q)^2)^{1/4} = 1.
\end{align*}
Therefore,
$$\int_R \big|2\partial_t^{1/2}W*(\nabla_x\varphi\,\nabla_x W*\nu)
+\partial_t^{1/2}W*((W*\nu-c)\,\Theta\varphi)\big|\,dm\lesssim\ell(R)^{n+2}.$$

So to prove the lemma it suffices to show that
$$\int_R\big|\partial_t^{1/2}\big[\varphi(W*\nu-W*\nu(\bar x_Q))\big]\big|\,dm\lesssim \ell(R)^{n+2}.$$
To this end, consider a $C^\infty$ function $\psi_Q$ such that $\chi_Q\leq \psi_Q\leq \chi_{2Q}$, with
$|\nabla_x\psi_Q|\lesssim1/\ell(Q)$ and $|\partial_t\psi_Q|\lesssim1/\ell(Q)^2$, and for any function
$F:\R^{n+1}\to\R$ consider the ``smooth mean" with respect to $\psi_Q$ defined by
$$m_{\psi_Q}(F) = \frac{\int F\,\psi_Q\,dm}{\int\psi_Q\,dm}.$$
Observe that for arbitrary functions $f,g:\R\to\R$, we have
$$\partial_t^{1/2}(f\,g)(t) = g(t)\,\partial_t^{1/2}f(t) + f(t)\,\partial_t^{1/2}g(t) 
+ \int \frac{\big(f(s)-f(t)\big) \big(g(s)-g(t)\big)}{|s-t|^{3/2}}\,ds.$$
Applying this with $f=\vphi(x,\cdot)$ and $g= W*\nu(x,\cdot)-W*\nu(\bar x_Q)$, we get
\begin{align*}
\partial_t^{1/2}\big[ &\varphi (W*\nu-W*\nu (\bar x_Q))\big](x,t)\\
 & =
\big(W*\nu(x,t)-W*\nu(\bar x_Q)\big) \,\partial_t^{1/2}\vphi(x,t) + \vphi(x,t)\,\partial_t^{1/2}W*\nu(x,t)\\
& \quad
+ \int \frac{\big(\vphi(x,s)-\vphi(x,t)\big) \,\big(W*\nu(x,s)-W*\nu(x,t)\big)}{|s-t|^{3/2}}\,ds\\
& =
\big(W*\nu(x,t)-W*\nu(\bar x_Q)\big) \,\partial_t^{1/2}\vphi(x,t) + \vphi(x,t)\,\big(\partial_t^{1/2}W*\nu(x,t) - m_{\psi_Q}(\partial_t^{1/2}W*\nu)\big)\\
& \quad
+ \int \frac{\big(\vphi(x,s)-\vphi(x,t)\big) \,\big(W*\nu(x,s)-W*\nu(x,t)\big)}{|s-t|^{3/2}}\,ds + 
\vphi(x,t)\,m_{\psi_Q}(\partial_t^{1/2}W*\nu)\\
& = A(\bar x)+B(\bar x)+C(\bar x)+D(\bar x).
\end{align*}

To estimate $A(\bar x)$, observe first that $\partial_t^{1/2}\vphi(x,t)$ vanishes unless $x\in Q_1$. 
In the case $x\in Q_1$, $t\in 2I_Q$, by the smoothness of $\vphi$ we have
$$|\partial_t^{1/2}\vphi(x,t)| \leq \int \frac{|\vphi(x,s) - \vphi(x,t)|}{|s-t|^{3/2}}\,ds
\lesssim \frac1{\ell(Q)^2} \int_{2I_Q} \frac{|s - t|}{|s-t|^{3/2}}\,ds + \int_{\R\setminus 2I_Q} \frac{1}{|s-t|^{3/2}}\,ds \lesssim \frac1{\ell(Q)}.$$
In the case $x\in Q_1$, $t\not\in 2I_Q$, we have
$$|\partial_t^{1/2}\vphi(x,t)| \leq \int \frac{|\vphi(x,s)|}{|s-t|^{3/2}}\,ds
\approx \frac1{|t-t_Q|^{3/2}}\, \int |\vphi(x,s)|\,ds \lesssim  \frac{\ell(Q)^2}{|t-t_Q|^{3/2}},$$
where $\bar x_Q=(x_Q,t_Q)$.
So in any case, 
$$|\partial_t^{1/2}\vphi(x,t)| \lesssim  \frac{\ell(Q)^2}{\ell(Q)^3+|t-t_Q|^{3/2}}.$$
Then, using the Lip and Lip $1/2$ conditions on $x$ and $t$ of $W*\nu$, we infer that, for $x\in Q_1$,
$$|A(\bar x)|\lesssim \frac{\ell(Q)^2 \,(\ell(Q) + |t-t_Q|^{1/2})}{\ell(Q)^3+|t-t_Q|^{3/2}}.$$
Therefore,
\begin{align*}
\int_R |A(\bar x)|\,d\bar x & \lesssim \int_{x\in Q_1}\int_{|t-t_Q|\leq 2\ell(R)^2} 
\frac{\ell(Q)^2 \,(\ell(Q) + |t-t_Q|^{1/2})}{\ell(Q)^3+|t-t_Q|^{3/2}}\,dt \\
& \lesssim \ell(Q)^{n+2}\,
\bigg(1+\log\frac{\ell(R)}{\ell(Q)}\bigg)\lesssim \ell(R)^{n+2}.
\end{align*}

To estimate the $L^1$ norm of the term $B$ we just use the fact that $\partial_t^{1/2}W*\nu$
is in the parabolic BMO space and that $|m_{\psi_Q}(\partial_t^{1/2}W*\nu) - m_{Q}(\partial_t^{1/2}W*\nu)
\big|\lesssim \|\partial_t^{1/2} W*\nu\|_{*,p} \leq 1$:
\begin{align*}
\int_R \big|B(\bar x)|\,d\bar x & = \int_R |\vphi(\bar x)\,\big(\partial_t^{1/2}W*\nu(\bar x) - m_{\psi_Q}(\partial_t^{1/2}W*\nu)\big)\big|\,d\bar x\\
& \lesssim \int_Q |\partial_t^{1/2}W*\nu(\bar x) - m_{\psi_Q}(\partial_t^{1/2}W*\nu)\big|\,d\bar x \lesssim \ell(Q)^{n+2}\leq \ell(R)^{n+2}.
\end{align*}

We are now left with $C(\bar x) + D(\bar x)$.
First we split
\begin{align*}
C(\bar x) & = \int \frac{\big(\vphi(x,s)-\vphi(x,t)\big) \,\big(W*\nu(x,s)-W*\nu(x,t)\big)}{|s-t|^{3/2}}\,ds
\\
&= \int_{|s-t|\leq \ell(Q)^2}\cdots + \int_{|s-t|>\ell(Q)^2}\cdots= C_1(\bar x) + C_2(\bar x).
\end{align*}
To estimate $C_1(\bar x)$ we use the smoothness of $\vphi$ and the Lip$(1/2)$ condition of $W*\nu$ in $t$:
$$|C_1(\bar x)|\lesssim \int_{|s-t|\leq \ell(Q)^2}  \frac{\ell(Q)^{-2}\,|s- t| \,|s-t|^{1/2}}{|s-t|^{3/2}}\,ds
\lesssim 1,$$
so that $\int_R |C_1(\bar x)|\,d\bar x\lesssim \ell(R)^{n+2}$.

Concerning $C_2(\bar x)$, we have
\begin{align*}
C_2(\bar x) & = 
\int_{|s-t|>\ell(Q)^2} \frac{W*\nu(x,s)-W*\nu(x,t)}{|s-t|^{3/2}}\,\vphi(x,s)\,ds\\
&\quad - \int_{|s-t|>\ell(Q)^2} \frac{W*\nu(x,s)-W*\nu(x,t)}{|s-t|^{3/2}}\,ds \,\vphi(x,t)\\
& = C_{2,1}(\bar x) - C_{2,2}(\bar x).
\end{align*}
Using again the Lip$(1/2)$ condition of $W*\nu$ in $t$, and the fact that 
$|\vphi(x,\cdot)|\lesssim \chi_{I_Q}$, we obtain
$$|C_{2,1}(\bar x)|\lesssim \int_{|s-t|>\ell(Q)^2} \frac1{|s-t|}\,\vphi(x,s)\,ds\lesssim \frac1{\ell(Q)^2}
\int\vphi(x,s)\,ds \lesssim 1,$$
and so $\int_R |C_{2,1}(\bar x)|\,d\bar x\lesssim \ell(R)^{n+2}$.

By the estimates above, we have
$$\int_R |\partial_t^{1/2} W*(\vphi\nu)|\,dm  \lesssim \ell(R)^{n+2} + 
\int_R |- C_{2,2}(\bar x) + D(\bar x)|\,d\bar x,$$
where
$$C_{2,2}(\bar x)= \int_{|s-t|>\ell(Q)^2} \frac{W*\nu(x,s)-W*\nu(x,t)}{|s-t|^{3/2}}\,ds \,\vphi(x,t),
\qquad D(\bar x) = 
m_{\psi_Q}(\partial_t^{1/2}W*\nu)\, \vphi(x,t).$$
So to conclude the prove of the lemma it suffices to show that
\begin{equation}\label{eqfii3}
\bigg|m_{\psi_Q}(\partial_t^{1/2}W*\nu) - \int_{|s-t|>\ell(Q)^2} \frac{W*\nu(x,s)-W*\nu(x,t)}{|s-t|^{3/2}}\,ds\bigg|\lesssim 1\quad \mbox{ for all $\bar x\in Q$.}
\end{equation}
To this end, first we turn our attention to the term $m_{\psi_Q}(\partial_t^{1/2}W*\nu)$. As above, for any
$\bar y = (y,u)\in\R^{n+1}$ we split
\begin{align*}
\partial_t^{1/2}W*\nu(\bar y) &= \!\int_{|s-u|\leq \ell(Q)^2}\!\!\! \!\!\frac{W*\nu(y,s) - W*\nu(y,u)}{|s-u|^{3/2}}\,ds + \int_{|s-u|> \ell(Q)^2}\!\! \!\!\!\frac{W*\nu(y,s) - W*\nu(y,u)}{|s-u|^{3/2}}\,ds\\
& =:F_1(\bar y) + F_2(\bar y).
\end{align*}
Observe that the kernel
$$K_y(s,u) = \chi_{|s-u|\leq \ell(Q)^2}\,\frac{W*\nu(y,s) - W*\nu(y,u)}{|s-u|^{3/2}},$$
is antisymmetric, and thus
\begin{align*}
m_{\psi_Q}F_1 &= \frac1{\|\psi_Q\|_1} \iiint K_y(s,u)\,\psi_Q(y,u)\,ds\,dy\,du \\& = 
-\frac1{\|\psi_Q\|_1} \iiint K_y(s,u)\,\psi_Q(y,s)\,du\,dy\,ds \\
&=\frac1{2\,\|\psi_Q\|_1} \iiint K_y(s,u)\,(\psi_Q(y,u)-\psi_Q(y,s))\,du\,dy\,ds.
\end{align*}
Hence, by the smoothness of $\psi_Q$ and the Lip $1/2$ condition of $W*\nu$ in $t$,
\begin{align*}
|m_{\psi_Q}F_1| & \leq
\frac1{2\,\|\psi_Q\|_1} \iiint_{|s-u|\leq \ell(Q)^2} \frac{|W*\nu(y,s) - W*\nu(y,u)|}{|s-u|^{3/2}}\,|\psi_Q(y,u)-\psi_Q(y,s)|\,du\,dy\,ds\\
&\lesssim 
\frac1{\ell(Q)^{n+2}} \iiint_{
\begin{subarray}{l} 
y\in 2Q_1\\
|s-u|\leq \ell(Q)^2\\ 
u\in 4I_Q\end{subarray}}
 \frac{|s-u|^{1/2}}{|s-u|^{3/2}}\,\frac{|u-s|}{\ell(Q)^2}\,du\,dy\,ds\lesssim 1.
\end{align*}

To prove \rf{eqfii3}, it remains to show that
\begin{equation}\label{eqdk34}
\bigg|m_{\psi_Q}F_2 - \int_{|s-t|>\ell(Q)^2} \frac{W*\nu(x,s)-W*\nu(x,t)}{|s-t|^{3/2}}\,ds\bigg|\lesssim 1\quad \mbox{ for all $\bar x\in Q$.}
\end{equation}
Clearly, it suffices to prove that
$$\bigg|F_2(\bar y) - \int_{|s-t|>\ell(Q)^2} \frac{W*\nu(x,s)-W*\nu(x,t)}{|s-t|^{3/2}}\,ds\bigg|
= \big|F_2(\bar y) - F_2(\bar x)\big|\lesssim 1\quad \mbox{ for all $\bar x\in Q$ and $\bar y\in 2Q$.}$$
We denote $A_t=\{s\in\R:|s-t|>\ell(Q)^2\}$, and analogously $A_u$. Then we split
\begin{align*}
\big|F_2(\bar y) - F_2(\bar x)\big| &\leq \int_{A_u\setminus A_t} 
\frac{|W*\nu(y,s)-W*\nu(y,u)|}{|s-u|^{3/2}}\,ds  + 
\int_{A_t\setminus  A_u} 
\frac{|W*\nu(x,s)-W*\nu(x,t)|}{|s-t|^{3/2}}\,ds \\
&\quad +\int_{A_u\cap A_t}  \bigg|\frac{W*\nu(x,s)-W*\nu(x,t)}{|s-t|^{3/2}} - \frac{W*\nu(y,s)-W*\nu(y,u)}{|s-u|^{3/2}}\bigg|
\,ds\\
& = I_1 + I_2 + I_3.
\end{align*}
Using the Lip $1/2$ condition of $W*\nu$ in $t$ and the fact that $|s-u|\approx\ell(Q)^2$ in 
$A_u\setminus A_t$ and $|s-t|\approx\ell(Q)^2$ in 
$A_t\setminus A_u$, it is immediate to check that
$$I_1 + I_2 \lesssim1.$$
Concerning $I_3$, by the triangle inequality,
\begin{align*}
I_3 & \leq 
\int_{A_u\cap A_t}  \bigg|\frac1{|s-t|^{3/2}} - \frac1{|s-u|^{3/2}}\bigg|\,\big|W*\nu(x,s)-W*\nu(x,t)\big| \,ds \\
& \quad +
\int_{A_u\cap A_t} \frac{\big|W*\nu(x,s)-W*\nu(x,t) - W*\nu(y,s) + W*\nu(y,u)\big|}{|s-u|^{3/2}}
\,ds\\
& = I_{3,1} + I_{3,2}.
\end{align*}
To estimate $I_{3,1}$ we take into account that 
$$\bigg|\frac1{|s-t|^{3/2}} - \frac1{|s-u|^{3/2}}\bigg|\lesssim \frac{|t-u|}{|s-t|^{5/2}}$$
in the domain of integration and we use the Lip $1/2$ condition on $W*\nu$:
$$I_{3,1}\lesssim \int_{|s-t|>\ell(Q)^2}  \frac{|t-u|}{|s-t|^{5/2}}\,|s-t|^{1/2} \,ds
\lesssim 1.$$
Finally we deal with $I_{3,2}$:
$$I_{3,2}\leq \int_{|s-t|>\ell(Q)^2} \frac{\big|W*\nu(x,s) - W*\nu(y,s)\big| + \big|W*\nu(y,u)-W*\nu(x,t)\big|}{|s-t|^{3/2}}
\,ds.$$
By the Lipschitz in $x$ and Lip $1/2$ in $t$ conditions of $W*\nu$, we derive
$$\big|W*\nu(x,s) - W*\nu(y,s)\big| + \big|W*\nu(y,u)-W*\nu(x,t)\big|\lesssim \ell(Q).$$
Therefore,
$$I_{3,2}\lesssim \int_{|s-t|>\ell(Q)^2} \frac{\ell(Q)}{|s-t|^{3/2}}
\,ds\lesssim1.$$
Together with the preceding estimates for $I_1$, $I_2$, $I_{3,1}$, this shows that
$$\big|F_2(\bar y) - F_2(\bar x)\big|\lesssim 1\quad \mbox{ for all $\bar x \in Q,$ and $\bar y\in 2Q$,}$$
which proves \rf{eqdk34} and concludes the proof of the lemma.
\end{proof}
\vv

\begin{lemma}\label{lembmo}
Let $Q\subset\R^{n+1}$ be a parabolic cube and let 
 $\nu$ be a distribution supported in $\R^{n+1}\setminus 4Q$ 
 with upper parabolic $1$-growth of degree $n+1$ and
 such that 
$$\|\nabla_x W*\nu\|_\infty \leq 1,\qquad \| W*\nu\|_{Lip_{1/2,t}} \leq 1.$$
Then,
$$\int_Q |\partial_t^{1/2} W*\nu - m_Q(\partial_t^{1/2} W*\nu)|\,dm  \lesssim \ell(Q)^{n+2}.$$
\end{lemma}

\begin{proof}
Let $Q\subset \R^{n+1}$ be a fixed parabolic cube.  To prove the lemma,  it is enough to show that
\begin{equation}
|(\partial_t^{1/2} W*\nu)(\bar x) - 
(\partial_t^{1/2} W*\nu)(\bar y)|\lesssim 1\label{eqdif1}
\end{equation}
for $\bar x,\bar y\in \R^{n+1}$ in the following two cases:
\begin{itemize}
\item Case 1:  $\bar x,\bar y\in Q$ of the form $\bar x= (x,t)$, $\bar y= (y,t)$.
\item Case 2:  $\bar x,\bar y\in Q$ of the form $\bar x= (x,t)$, $\bar y= (x,u)$.
\end{itemize}

\vv
\noi {\em Proof of \rf{eqdif1} in Case 1.}
We split
\begin{align*}
|&\partial_t^{1/2} W* \nu(x,t) - \partial_t^{1/2} W* \nu(y,t)|\\
 & =
\left|\int \frac{W*\nu(x,s) - W*\nu(x,t)}{|s-t|^{3/2}}\,ds -\!
\int \frac{W*\nu(y,s) - W*\nu(y,t)}{|s-t|^{3/2}}\,ds\right|\\
& \leq
\int_{|s-t|\leq4\ell(Q)^2} \frac{|W*\nu(x,s) - W*\nu(x,t)|}{|s-t|^{3/2}}\,ds\\
&\quad+\int_{|s-t|\leq4\ell(Q)^2} \frac{|W*\nu(y,s) - W*\nu(y,t)|}{|s-t|^{3/2}}\,ds
\\
&\quad+\!\int_{|s-t|>4\ell(Q)^2}\!\!\!\! \frac{|W\!*\nu(x,s) - W\!*\nu(x,t) - W\!*\nu(y,s) + W\!*\nu(y,t)|}{|s-t|^{3/2}}\,ds
\\
& =: A_1  + A_2 + B.
\end{align*}

We will estimate the term $A_1$ now. For $s,t$ such that $|s-t|\leq4\ell(Q)^2$, we write
$$|W*\nu(x,s) - W*\nu(x,t)|\leq |s-t|\,\|\partial_t W*\nu\|_{\infty,3Q}.$$
We claim that 
\begin{equation}\label{claim2}
\|\partial_t W*\nu\|_{\infty,3Q}\lesssim \frac1{\ell(Q)}.
\end{equation}
Once claim \eqref{claim2} is proved, we get that
 $$|W*\nu(x,s) - W*\nu(x,t)|\lesssim \frac{|s-t|}{\ell(Q)}.$$
 Plugging this into the integral that
defines $A_1$, we obtain
$$A_1\lesssim \int_{|s-t|\leq4\ell(Q)^2} \frac{|s-t|}{\ell(Q)\,|s-t|^{3/2}}\,ds 
\lesssim \frac{(\ell(Q)^2)^{1/2}}{\ell(Q)}=1.$$
By exactly the same arguments, just writing $y$ in place of $x$ above, we deduce also that
$$A_2\lesssim 1.$$
Concerning the term $B$, we write
\begin{align*}
B &\leq 
\int_{|s-t|>4\ell(Q)^2} \frac{|W*\nu(x,s) - W*\nu(y,s)|}{|s-t|^{3/2}}\,ds
+
\int_{|s-t|>4\ell(Q)^2} \frac{|W*\nu(x,t) - W*\nu(y,t)|}{|s-t|^{3/2}}\,ds.
\end{align*}
Then,
$$|W*\nu(x,s) - W*\nu(y,s)|\leq \|\nabla_x W*\nu\|_\infty 
\,|x-y|\lesssim \ell(Q).$$
The same estimate holds replacing $(x,s)$ and $(y,s)$ by $(x,t)$ and $(y,t)$.
Hence,
$$B\lesssim \int_{|s-t|>4\ell(Q)^2} \frac{\ell(Q)}{|s-t|^{3/2}}\,ds \lesssim \frac{\ell(Q)}{(\ell(Q)^2)^{1/2}}\lesssim 1.$$
So, once claim \rf{claim2} is proved, \rf{eqdif1} holds in Case 1.

To show \rf{claim2}, we split $\R^{n+1}\setminus 4Q$ into parabolic annuli $A_k=2^{k+1}Q\setminus 2^kQ$ and consider $C^2$ functions $\wt\chi_k$, supported on $\frac32A_k$ which equal 1 on $A_k$, 
vanish on $(\frac32A_k)^c$ and satisfy
$$\sum_{k\geq 3}\wt\chi_k  = 1 \quad\mbox{ in $\R^{n+1}\setminus 4Q$}$$
and
$$\|\nabla_x\wt\chi_k\|_\infty\lesssim\frac1{2^k\ell(Q)}, \qquad
\|\nabla^2_x \wt\chi_k\|_\infty + \|\partial_t \wt\chi_k\|_\infty 
\lesssim \frac1{(2^k\ell(Q))^2}.$$
Then, for each $\bar z=(z,v)\in 3Q$, $$|\partial_t W*\nu(\bar z)|\leq\sum_{k\geq 3}|\partial_t W*(\wt\chi_k\nu)(\bar z)|.$$ Claim \rf{claim2} will be proved if we show that for each $k\geq 2$,
\begin{equation}\label{k}
|\partial_t W*(\wt\chi_k\nu)(\bar z)|\lesssim\frac{2^{-k}}{\ell(Q)}.
\end{equation}
Write $$|\partial_t W*(\wt\chi_k\nu)(\bar z)|=|\langle \nu, \wt\chi_k\partial_t W(\bar z-\cdot)\rangle|=|\langle \nu, \psi_k\rangle|,$$ the last equality being a definition of $\psi_k $. 

To estimate \rf{k}, 
we want to use the upper parabolic growth of $\nu$.
Therefore we have to study the admissibility conditions \rf{admis} of 
$\psi_k$ for each $k$. This means we have to estimate the norms $\|\nabla_x\psi_k\|_\infty$ and $\|\Delta\psi_k\|_{\infty}+\|\partial_t\psi_k\|_\infty$.  Write
\begin{equation}\label{prod}
\nabla_x\psi_k=\nabla_x\wt\chi_k\,\partial_tW(\bar z-\cdot)+\nabla_x\partial_tW(\bar z-\cdot)\,\wt\chi_k.
\end{equation}
The estimate of the $L^\infty-$norm of the first term in \rf{prod} comes from  $\|\nabla_x\wt\chi_k\|_\infty\lesssim(2^k\ell(Q))^{-1}$ and Lemma \ref{estimates}, together with the fact that for $\bar x\in Q$ and $\bar z\in A_k$ we have
\begin{equation}\label{partialtW}
|\partial_t W(\bar x-\bar z)|\lesssim\frac1{|\bar x-\bar z|_p^{n+2}}\approx\frac1{(2^k\ell(Q))^{n+2}}.
\end{equation}
For the second term in \rf{prod} we have to compute $\nabla_x\partial_tW$. Arguing as in the proof of Lemma \ref{estimates} one can show that
\begin{equation}\label{nablapartialtW}
|\nabla_x\partial_tW(\bar x)|\lesssim\frac1{\max(t^{(n+3)/2},|x|^{n+3})}\approx\frac1{|\bar x|_p^{n+3}}.
\end{equation}
Putting these estimates together we get $$\|\nabla_x\psi_k\|_\infty\lesssim\frac1{(2^k\ell(Q))^{n+3}}.$$
To estimate $\|\Delta\psi_k\|_{\infty}$ write 
\begin{equation}
\Delta\psi_k=\Delta\wt\chi_k\partial_tW(\bar z-\cdot)+2\nabla_x\wt\chi_k\nabla_x\partial_tW(\bar z-\cdot)+\wt\chi_k\Delta\partial_tW(\bar z-\cdot).
\end{equation}
Following the proof of Lemma \ref{estimates} one can deduce that 
\begin{equation}\label{DeltaW}
|\Delta\partial_tW(\bar x)|\lesssim\frac1{\max(t^{(n+4)/2},|x|^{n+4})}\approx\frac1{|\bar x|_p^{n+4}}.
\end{equation}
Hence, using the estimates $\|\Delta\wt\chi_k\|_\infty\lesssim(2^k\ell(Q))^{-2}$, $\|\nabla_x\wt\chi_k\|_\infty\lesssim(2^k\ell(Q))^{-1}$, \eqref{nablapartialtW} and \eqref{DeltaW} one obtains
$$\|\Delta\psi_k\|_{\infty}\lesssim\frac1{(2^k\ell(Q))^{n+4}}.$$
To estimate $\|\partial_t\psi_k\|_\infty$ write
\begin{equation*}
\partial_t\psi_k=\partial_t\wt\chi_k\,\partial_tW(\bar z-\cdot)+\partial^2_tW(\bar z-\cdot)\,\wt\chi_k.
\end{equation*}
The first term above is estimated by using $\|\partial_t\wt\chi_k\|_\infty\leq(2^{k}\ell(Q))^{-2}$ and \rf{partialtW}. For the second term we argue as in the proof of Lemma \ref{estimates} and obtain
$$|\partial^2_t W(\bar x-\bar z)|\lesssim\frac1{|\bar x-\bar z|_p^{n+4}}\approx\frac1{(2^k\ell(Q))^{n+4}},$$
for $\bar x\in Q$ and $\bar z\in A_k$. Therefore $$\|\partial_t\psi_k\|_\infty\lesssim\frac1{(2^k\ell(Q))^{n+4}}.$$
Hence, by Lemma \ref{lemgrow},
$$|(\partial_t W*\wt\chi_k\nu)(\bar z)|=\frac1{(2^k\ell(Q))^{n+2}}|\langle \nu, (2^k\ell(Q))^{n+2}\psi_k\rangle|\lesssim\frac{(2^k\ell(Q))^{n+1}}{(2^k\ell(Q))^{n+2}}=\frac{2^{-k}}{\ell(Q)},$$
which concludes the proof of claim \rf{claim2} and of \rf{eqdif1} in Case 1.

\vv
\noi {\em Proof of \rf{eqdif1} in Case 2.}
As in Case 1 we write
\begin{align*}
|&\partial_t^{1/2} W*\nu(x,t) - \partial_t^{1/2} W*\nu(x,u)|\\
 & =
\left|\int \frac{W*\nu(x,s) - W*\nu(x,t)}{|s-t|^{3/2}}\,ds -\!
\int \frac{W*\nu(x,s) - W*\nu(x,u)}{|s-u|^{3/2}}\,ds\right|\\
& \leq
\int_{|s-t|\leq4\ell(Q)^2} \frac{|W*\nu(x,s) - W*\nu(x,t)|}{|s-t|^{3/2}}\,ds\\
&\quad+
\int_{|s-t|\leq4\ell(Q)^2} \frac{|W*\nu(x,s) - W*\nu(x,u)|}{|s-u|^{3/2}}\,ds
\\
&\quad+\int_{|s-t|>4\ell(Q)^2} \bigg|\frac{W*\nu(x,s) - W*\nu(x,t)}{|s-t|^{3/2}}-\frac{W*\nu(x,s) - W*\nu(x,u)}{|s-u|^{3/2}}\bigg| ds\\
& =: A_1'  + A_2' + B'.
\end{align*}
The terms $A_1'$ and $A_2'$ can be estimated exactly in the same way as the terms $A_1$ and $A_2$ in Case 1, so that
$$A_1' + A_2'\lesssim 1.$$

Concerning $B'$ we have
\begin{align*}
B' & \leq \int_{|s-t|>4\ell(Q)^2} \bigg|\frac{W*\nu(x,s) - W*\nu(x,t)}{|s-t|^{3/2}}-\frac{W*\nu(x,s) - W*\nu(x,u)}{|s-u|^{3/2}}\bigg| ds\\
&\leq 
 \int_{|s-t|>4\ell(Q)^2} 
\left|\frac1{|s-t|^{3/2}} - \frac1{|s-u|^{3/2}}\right|\,
 \big|W*\nu(x,s) - W*\nu(x,t)\big| \,ds\\
&\quad 
+ \int_{|s-t|>4\ell(Q)^2} 
\frac1{|s-u|^{3/2}}\,
 \big|W*\nu(x,t) - W*\nu(x,u)\big| \,ds.
\end{align*}
Taking into account that, for $|s-t|>4\ell(Q)^2$,
$$\left|\frac1{|s-t|^{3/2}} - \frac1{|s-u|^{3/2}}\right|\lesssim \frac{|t-u|}{|s-t|^{5/2}}
\lesssim \frac{\ell(Q)^2}{|s-t|^{5/2}}$$
and that  $ \| W*\nu\|_{Lip_{1/2,t}} \lesssim 1$,
we deduce 
$$B' \lesssim 
 \int_{|s-t|>4\ell(Q)^2} \frac{\ell(Q)^2}{|s-t|^{5/2}}
\,|s - t|^{1/2} \,ds\\
+ 
 \int_{|s-t|>4\ell(Q)^2}\frac1{|s-u|^{3/2}}\,|t-u|^{1/2} \,ds \lesssim 1.
$$
\end{proof}

\vv

We will also need the following (easy) technical result.

\begin{lemma}\label{lemcre3} 
Let $\nu$ be a distribution in $\R^{n+1}$ which has upper parabolic $1$-growth of degree $n+1$. Let $\vphi$ be a $C^2$ function admissible for some parabolic cube $Q$. Then $\vphi\nu$ has upper $C$-parabolic growth of degree $n+1$, for some absolute constant $C>0$
\end{lemma}

\begin{proof}
Let $\psi$ be a $C^2$ function admissible for some parabolic cube $R$. In the case $\ell(Q)\leq \ell(R)$, let $S=Q$, and otherwise let $S=R$.
 It is easy to check that
$c\vphi\psi$ is admissible for $S$, for some absolute constant $c>0$, and thus
$$|\langle \vphi\nu,\psi\rangle| = |\langle \nu,\vphi\psi\rangle|\lesssim \ell(S)^{n+1}\leq \ell(R)^{n+1}.$$
\end{proof}

\vv

The next lemma, together with Lemma \ref{lemlocnabla}, completes the proof of Theorem \ref{teoloc}.

\begin{lemma}\label{lemlocdt}
Let $\nu$ be a distribution in $\R^{n+1}$ such that 
$$\|\nabla_x W*\nu\|_\infty \leq 1,\qquad \|\partial_t^{1/2} W*\nu\|_{*,p} \leq 1.$$
Then, if $\vphi$ is a $C^2$ function admissible for a parabolic cube $Q$, we have
$$\|\partial_t^{1/2} W*(\vphi\nu)\|_{*,p} \leq 1.$$
\end{lemma}

\begin{proof}
Let $R\subset\R^{n+1}$ be some parabolic cube. We have to show that
there exists some constant $c_R$
(to be chosen below) such that
$$\int_R |\partial_t^{1/2} W*(\vphi\nu) - c_R|\,dm \lesssim \ell(R)^{n+2}.$$
To this end, we consider a $C^2$ function $\wt\chi_{5R}$ which equals $1$ on $5R$, vanishes in $6R^c$,
and satisfies
$$\|\nabla_x \wt\chi_{5R}\|_\infty \lesssim \frac1{\ell(R)}, \qquad
\|\nabla^2_x \wt\chi_{5R}\|_\infty + \|\partial_t \wt\chi_{5R}\|_\infty 
\lesssim \frac1{\ell(R)^2}.
$$
We also denote $\wt\chi_{5R^c} = 1 - \wt\chi_{5R}$. Then we
split
$$\int_R |\partial_t^{1/2} W*(\vphi\nu) - c_R|\,dm \leq
\int_R |\partial_t^{1/2} W*(\wt\chi_{5R}\vphi\nu)|\,dm + \int_R |\partial_t^{1/2} W*(\wt\chi_{5R^c}\vphi\nu) - c_R|\,dm =: I_1 + I_2.
$$

To estimate $I_2$ we intend to apply Lemma \ref{lembmo}. Notice that $\supp(\wt\chi_{5R^c}\vphi\nu)\subset \overline{5R^c}$. We claim that
\begin{equation}\label{eqclaim1}
\|\nabla_x W*(\wt\chi_{5R^c}\vphi\nu)\|_\infty \lesssim 1,\qquad
\| W*(\wt\chi_{5R^c}\vphi\nu)\|_{Lip_{1/2,t}} \lesssim 1.
\end{equation}
To check this, just write
$$W*(\wt\chi_{5R^c}\vphi\nu) = W*(\vphi\nu) - W*(\wt\chi_{5R}\vphi\nu).$$
Since $\vphi$ is admissible for $Q$, we have
\begin{equation}\label{eqclaim2}
\|\nabla_x W*(\vphi\nu)\|_\infty \lesssim 1,\qquad
\| W*(\vphi\nu)\|_{Lip_{1/2,t}} \lesssim 1.
\end{equation}
Also, in case that $\ell(R)\leq \ell(Q)$, it is easy to check that there exists some absolute constant $c>0$ such that $c\wt\chi_{5R}\vphi$ 
is admissible for $5R$. On the other hand, if $\ell(R)> \ell(Q)$, then  $c\wt\chi_{5R}\vphi$ 
is admissible for $Q$, for some absolute constant $c>0$. So in any case, by Lemmas \ref{lemlocnabla}
and \ref{lem3.3}, 
$$\|\nabla_x W*(\wt\chi_{5R}\vphi\nu)\|_\infty \lesssim 1,\qquad
\| W*(\wt\chi_{5R}\vphi\nu)\|_{Lip_{1/2,t}} \lesssim 1.
$$
Hence, \rf{eqclaim1} follows from \rf{eqclaim2} and the preceding estimates.

On the other hand, by Lemma \ref{lemgrow}, $\nu$ has upper parabolic growth of degree $n+1$, and 
since  $c\wt\chi_{5R}\vphi$ is admissible either for $5R$ or for $Q$,  $\wt\chi_{5R}\vphi\nu$ also has upper parabolic $C$-growth for some absolute constant $C$, by Lemma
\ref{lemcre3}.
Then, from Lemma \ref{lembmo}, choosing $c_R = m_R\big(\partial_t^{1/2} W*(\wt\chi_{5R^c}\vphi\nu)\big)$, 
 we deduce that
$$I_2\lesssim \ell(R)^{n+2}.$$

To estimate $I_1$,  we may assume that $Q\cap 6R\neq\varnothing$, since otherwise $\wt\chi_{5R}\,\vphi
\equiv0$. Next we distinguish two cases. First we assume that $\ell(R)\leq \ell(Q)$, so that 
$c\wt\chi_{5R}\,\vphi$ is admissible for $5R$. Then, from Lemma \ref{keylemma} we derive
$$I_1\leq \int_{5R} |\partial_t^{1/2} W*(\wt\chi_{5R}\vphi\nu)|\,dm\lesssim \ell(R)^{n+2}.$$
In the case $\ell(R)> \ell(Q)$, the fact that $Q\cap 6R\neq\varnothing$ implies that $Q\subset8R$,
and $c\wt\chi_{5R}\,\vphi$ is admissible for $Q$, for some $c\approx1$. 
Then again from Lemma \ref{keylemma} we infer that
$$I_1\leq \int_{8R} |\partial_t^{1/2} W*(\wt\chi_{5R}\vphi\nu)|\,dm\lesssim \ell(R)^{n+2}.$$
Together with the estimate obtained for $I_2$, this proves the lemma.
\end{proof}

\vv

\section{The case when $\nu$ is a positive measure}\label{section-positive}

The main goal of this section will be to prove the following result.

\begin{lemma}\label{propo1}
Let $\mu$ be a measure in $\R^{n+1}$ which has upper parabolic growth of degree $n+1$ with constant $1$ such that
$$\|\nabla_x W*\mu \|_\infty\leq 1.$$
Then
$$\|\partial_t^{1/2} W * \mu\|_{*,p}\lesssim 1.$$
\end{lemma}

\vv
For that we need the following lemma.
\begin{lemma} \label{lemlip1/2}
Let $\mu$ be a measure in $\R^{n+1}$ which has upper parabolic growth of degree $n+1$ with constant $1$. Then,
$$\|W*\mu\|_{Lip_{1/2,t}}\lesssim 1.$$
\end{lemma}

\begin{proof}
Let $\bar x = (x,t)$, $\hat x = (x,u)$, and $\bar x_0 = \frac12(\bar x + \hat x)$.
Then, writing $\bar y=(y,s)$, we split
\begin{align*}
|W*\mu(\bar x) - W*\mu(\hat x)| & \leq \int_{|\bar y -\bar x_0|_p\geq 2 |\bar x-\hat x|_p}
|W(x-y,t-s) - W(x-y,u-s)|\,d\mu(\bar y) \\
&\quad + \int_{|\bar y -\bar x_0|_p < 2 |\bar x-\hat x|_p}\!
|W(x-y,t-s) - W(x-y,u-s)|\,d\mu(\bar y) =: I_1 + I_2.
\end{align*}

To shorten notation, we write $d:=|\bar x-\hat x|_p = |t-u|^{1/2}$. Then we have
$$I_1\lesssim \sum_{k\geq 1} \int_{2^k d\leq |\bar y -\bar x_0|_p< 2^{k+1}d} \,\,\sup_{\xi\in [\bar x-\bar y,
\hat x-\bar y]} |\partial_t W (\xi)| \,|t-u|\,d\mu(\bar y).$$
Since 
$$|\partial_t W (\xi)|\lesssim \frac1{|\xi|_p^{n+2}}\approx \frac1{(2^kd)^{n+2}}\quad
\mbox{ if $\xi\in [\bar x-\bar y,
\hat x-\bar y]$, $|\bar y -\bar x_0|_p\approx 2^{k}d$,}$$
we deduce that
$$I_1\lesssim \sum_{k\geq 1} \frac{\mu(B_p(\bar x_0,2^{k+1}d))}{(2^k d)^{n+2}}\,|t-u|\lesssim \frac{|t-u|}d= |t-u|^{1/2}.$$

Next we deal with $I_2$. Writing $B_0= B_p(\bar x_0,2d)$, we have
$$I_2 \leq W*(\chi_{B_0}\mu)(\bar x) + W*(\chi_{B_0}\mu)(\hat x).$$
Observe now that
$$0\leq W*(\chi_{B_0}\mu)(\bar x) \lesssim \int_{\bar y\in B_0} \frac1{|\bar x- \bar y|_p^n}\,
d\mu(\bar y) 
\leq \int_{|\bar x-\bar y|\leq 4d} \frac1{|\bar x- \bar y|_p^n}\,
d\mu(\bar y)
\lesssim d =|t-u|^{1/2}.$$
The last estimate follows by splitting the integral into parabolic annuli and using the parabolic
growth of order $n+1$ of $\mu$, for example.
The same estimate holds replacing $\bar x$ by $\hat x$. Then gathering all the estimates above, the lemma
follows.
\end{proof}

\vv
\begin{proof}[Proof of Lemma \ref{propo1}]
Let $Q\subset \R^{n+1}$ be a fixed parabolic cube. We have to show that there exists some constant $c_Q$
(to be chosen below) such that
$$\int_Q |\partial_t^{1/2} W*\mu - c_Q|\,dm \lesssim \ell(Q)^{n+2}.$$
To this end, we consider a $C^2$ function $\wt\chi_{5Q}$ which equals $1$ on $5Q$, vanishes in $6Q^c$,
and satisfies
$$\|\nabla_x \wt\chi_{5Q}\|_\infty \lesssim \frac1{\ell(Q)}, \qquad
\|\nabla^2_x \wt\chi_{5Q}\|_\infty + \|\partial_t \wt\chi_{5Q}\|_\infty 
\lesssim \frac1{\ell(Q)^2}.
$$
We also denote $\wt\chi_{5Q^c} = 1 - \wt\chi_{5Q}$. Then we
split
$$\int_Q |\partial_t^{1/2} W*\mu - c_Q|\,dm \leq
\int_Q |\partial_t^{1/2} W*(\wt\chi_{5Q}\mu)|\,dm + \int_Q |\partial_t^{1/2} W*(\wt\chi_{5Q^c}\mu) - c_Q|\,dm =: I_1 + I_2.
$$
To deal with the integral $I_1$, we just write
$$
\int_Q |\partial_t^{1/2} W*(\wt\chi_{5Q}\mu)|\,dm\lesssim
\int_Q \frac1{|x|^{n-1}\,|\bar x|_p^2}* (\wt\chi_{5Q}\mu)\,dm \leq \int_{6Q} \frac1{|y|^{n-1}\,|\bar y|_p^2}* (\chi_{Q}m)\,d\mu.$$
Taking into account that, for $\bar y=(y,u)$,
$$\frac1{|y|^{n-1}\,|\bar y|_p^2}\leq \frac1{|y|^{n-1/2}}\,\frac1{u^{3/4}},$$
and writing $Q=Q_1\times I_Q$, where $Q_1$ is a cube with side length $\ell(Q)$ in $\R^n$ and $I_Q$ is an interval of length $\ell(Q)^2$,
we deduce that for $\bar x=(x,t)\in 6Q$,
$$\frac1{|y|^{n-1}\,|\bar y|_p^2}* (\chi_{Q}m)(\bar x) \lesssim \int_{y\in Q_1} \frac1{|x-y|^{n-1/2}}dy\,\int_{u\in
I_Q}\frac1{|t-u|^{3/4}}\,du \lesssim \ell(Q)^{1/2}\,(\ell(Q)^2)^{1/4} = \ell(Q).$$
Thus,
$$\int_Q |\partial_t^{1/2} W*(\wt\chi_{5Q}\mu)|\,dm\lesssim \ell(Q)\,\mu(6Q)\lesssim \ell(Q)^{n+2}.$$

Next we will estimate the integral $I_2$, taking $c_Q:=
\partial_t^{1/2} W*(\wt\chi_{5Q^c}\mu)(\bar x_Q)$, where $\bar x_Q$ is the center of $Q$.
We follow the same scheme as in the proof of Lemma \ref{lembmo}.
To show that $I_2\lesssim \ell(Q)^{n+2}$, it suffices to prove
$$|\partial_t^{1/2} W*(\wt\chi_{5Q^c}\mu)(\bar x) - 
\partial_t^{1/2} W*(\wt\chi_{5Q^c}\mu)(\bar x_Q)|\lesssim 1.$$
In turn, to prove this it is enough to show that
\begin{equation}
|\partial_t^{1/2} W*(\wt\chi_{5Q^c}\mu)(\bar x) - 
\partial_t^{1/2} W*(\wt\chi_{5Q^c}\mu)(\bar y)|\lesssim 1\label{eqdif}
\end{equation}
for $\bar x,\bar y\in \R^{n+1}$ in the following two cases:
\begin{itemize}
\item Case 1:  $\bar x,\bar y\in Q$ of the form $\bar x= (x,t)$, $\bar y= (y,t)$.
\item Case 2:  $\bar x,\bar y\in Q$ of the form $\bar x= (x,t)$, $\bar y= (x,u)$.
\end{itemize}

\vv
\noi {\em Proof of \rf{eqdif} in Case 1.}
We split
\begin{align*}
|&\partial_t^{1/2} W* (\wt\chi_{5Q^c}\mu)(x,t) - \partial_t^{1/2} W* (\wt\chi_{5Q^c}\mu)(y,t)|\\
 & =
\left|\int \frac{W*(\wt\chi_{5Q^c}\mu)(x,s) - W*(\wt\chi_{5Q^c}\mu)(x,t)}{|s-t|^{3/2}}\,ds -\!
\int \frac{W*(\wt\chi_{5Q^c}\mu)(y,s) - W*(\wt\chi_{5Q^c}\mu)(y,t)}{|s-t|^{3/2}}\,ds\right|\\
& \leq
\int_{|s-t|\leq4\ell(Q)^2} \frac{|W*(\wt\chi_{5Q^c}\mu)(x,s) - W*(\wt\chi_{5Q^c}\mu)(x,t)|}{|s-t|^{3/2}}\,ds\\
&\quad  +
\int_{|s-t|\leq4\ell(Q)^2} \frac{|W*(\wt\chi_{5Q^c}\mu)(y,s) - W*(\wt\chi_{5Q^c}\mu)(y,t)|}{|s-t|^{3/2}}\,ds
\\
&\quad +
\!\int_{|s-t|>4\ell(Q)^2}\!\!\!\! \frac{|W\!*(\wt\chi_{5Q^c}\mu)(x,s) - W\!*(\wt\chi_{5Q^c}\mu)(x,t) - W\!*(\wt\chi_{5Q^c}\mu)(y,s) + W\!*(\wt\chi_{5Q^c}\mu)(y,t)|}{|s-t|^{3/2}}\,ds\\
& =: A_1  + A_2 + B.
\end{align*}

First we will estimate the term $A_1$ . For $s,t$ such that $|s-t|\leq4\ell(Q)^2$, we write
$$|W*(\wt\chi_{5Q^c}\mu)(x,s) - W*(\wt\chi_{5Q^c}\mu)(x,t)|\leq |s-t|\,\|\partial_t W*(\wt\chi_{5Q^c}\mu)\|_{\infty,3Q}.$$
Observe now that for each $\bar z= (z,v)\in 3Q$,
$$|\partial_t W*(\wt\chi_{5Q^c}\mu)(\bar z)| \lesssim\int_{5Q^c} \frac1{|\bar z-\bar w|_p^{n+2}}\,d\mu(\bar w)
\lesssim \frac1{\ell(Q)},$$
by splitting the last domain of integration into parabolic annuli and using the growth condition of order $n+1$ of $\mu$. Thus,
 $$|W*(\wt\chi_{5Q^c}\mu)(x,s) - W*(\wt\chi_{5Q^c}\mu)(x,t)|\lesssim \frac{|s-t|}{\ell(Q)}.$$
 Plugging this into the integral that
defines $A_1$, we obtain
$$A_1\lesssim \int_{|s-t|\leq4\ell(Q)^2} \frac{|s-t|}{\ell(Q)\,|s-t|^{3/2}}\,ds 
\lesssim \frac{(\ell(Q)^2)^{1/2}}{\ell(Q)}=1.$$
By exactly the same arguments, just writing $y$ in place of $x$ above, we deduce also that
$$A_2\lesssim 1.$$

Concerning the term $B$, we write
\begin{align*}
B & \leq \!\int_{|s-t|>4\ell(Q)^2}\!\!\!\! \frac{|W\!*(\wt\chi_{5Q^c}\mu)(x,s) - W\!*(\wt\chi_{5Q^c}\mu)(x,t) - W\!*(\wt\chi_{5Q^c}\mu)(y,s) + W\!*(\wt\chi_{5Q^c}\mu)(y,t)|}{|s-t|^{3/2}}\,ds\\
&
\leq 
\int_{|s-t|>4\ell(Q)^2} \frac{|W*(\wt\chi_{5Q^c}\mu)(x,s) - W*(\wt\chi_{5Q^c}\mu)(y,s)|}{|s-t|^{3/2}}\,ds\\
&\quad
+
\int_{|s-t|>4\ell(Q)^2} \frac{|W*(\wt\chi_{5Q^c}\mu)(x,t) - W*(\wt\chi_{5Q^c}\mu)(y,t)|}{|s-t|^{3/2}}\,ds.
\end{align*}
By Lemma \ref{lemlocnabla}, it follows that $\|\nabla_x W*(\wt\chi_{5Q}\mu)\|_\infty\lesssim 1$, and thus
$$\|\nabla_x W*(\wt\chi_{5Q^c}\mu)\|_\infty \leq \|\nabla_x W*\mu\|_\infty
+ \|\nabla_x W*(\wt\chi_{5Q}\mu)\|_\infty\lesssim 1.$$
Therefore,
$$|W*(\wt\chi_{5Q^c}\mu)(x,s) - W*(\wt\chi_{5Q^c}\mu)(y,s)|\leq \|\nabla_x W*(\wt\chi_{5Q^c}\mu)\|_\infty 
\,|x-y|\lesssim \ell(Q).$$
The same estimate holds replacing $(x,s)$ and $(y,s)$ by $(x,t)$ and $(y,t)$.
Hence,
$$B\lesssim \int_{|s-t|>\ell(Q)^2} \frac{\ell(Q)}{|s-t|^{3/2}}\,ds \lesssim \frac{\ell(Q)}{(\ell(Q)^2)^{1/2}}\lesssim 1.$$
So \rf{eqdif} holds in this case.

\vv
\noi {\em Proof of \rf{eqdif} in Case 2.}
As in Case 1, we write
\begin{align*}
|&\partial_t^{1/2} W* (\wt\chi_{5Q^c}\mu)(x,t) - \partial_t^{1/2} W* (\wt\chi_{5Q^c}\mu)(x,u)|\\
 & =
\left|\int \frac{W*(\wt\chi_{5Q^c}\mu)(x,s) - W*(\wt\chi_{5Q^c}\mu)(x,t)}{|s-t|^{3/2}}\,ds -\!
\int \frac{W*(\wt\chi_{5Q^c}\mu)(x,s) - W*(\wt\chi_{5Q^c}\mu)(x,u)}{|s-u|^{3/2}}\,ds\right|\\
& \leq
\int_{|s-t|\leq4\ell(Q)^2} \frac{|W*(\wt\chi_{5Q^c}\mu)(x,s) - W*(\wt\chi_{5Q^c}\mu)(x,t)|}{|s-t|^{3/2}}\,ds\\
&\quad  +
\int_{|s-t|\leq4\ell(Q)^2} \frac{|W*(\wt\chi_{5Q^c}\mu)(x,s) - W*(\wt\chi_{5Q^c}\mu)(x,u)|}{|s-u|^{3/2}}\,ds
\\
&\quad +
\biggl|\int_{|s-t|>4\ell(Q)^2} \frac{W*(\wt\chi_{5Q^c}\mu)(x,s) - W*(\wt\chi_{5Q^c}\mu)(x,t)}{|s-t|^{3/2}}\,ds \\
&\qquad\qquad\qquad -\int_{|s-t|>4\ell(Q)^2} \frac{W*(\wt\chi_{5Q^c}\mu)(x,s) - W*(\wt\chi_{5Q^c}\mu)(x,u)}{|s-u|^{3/2}}\,ds\biggr|\\
& =: A_1'  + A_2' + B'.
\end{align*}
The terms $A_1'$ and $A_2'$ can be estimated exactly in the same way as the terms $A_1$ and $A_2$ in Case 1, so that
$$A_1' + A_2'\lesssim 1.$$

Concerning $B'$, we have
\begin{align*}
B' & \leq \int_{|s-t|>4\ell(Q)^2} \bigg|\frac{W*(\wt\chi_{5Q^c}\mu)(x,s) - W*(\wt\chi_{5Q^c}\mu)(x,t)}{|s-t|^{3/2}}\\
& \quad- \frac{W*(\wt\chi_{5Q^c}\mu)(x,s) - W*(\wt\chi_{5Q^c}\mu)(x,u)}{|s-u|^{3/2}}\bigg| ds\\
&\leq 
 \int_{|s-t|>4\ell(Q)^2} 
\left|\frac1{|s-t|^{3/2}} - \frac1{|s-u|^{3/2}}\right|\,
 \big|W*(\wt\chi_{5Q^c}\mu)(x,s) - W*(\wt\chi_{5Q^c}\mu)(x,t)\big| \,ds\\
&\quad 
+ \int_{|s-t|>4\ell(Q)^2} 
\frac1{|s-u|^{3/2}}\,
 \big|W*(\wt\chi_{5Q^c}\mu)(x,t) - W*(\wt\chi_{5Q^c}\mu)(x,u)\big| \,ds
\end{align*}
Taking into account that, for $|s-t|>4\ell(Q)^2$,
$$\left|\frac1{|s-t|^{3/2}} - \frac1{|s-u|^{3/2}}\right|\lesssim \frac{|t-u|}{|s-t|^{5/2}}
\lesssim \frac{\ell(Q)^2}{|s-t|^{5/2}}$$
and that, by Lemma \ref{lemlip1/2}, $W*(\wt\chi_{5Q^c}\mu)(x,\cdot)$ is Lip$(1/2)$ in the variable $t$, 
we deduce that
$$B' \lesssim 
 \int_{|s-t|>4\ell(Q)^2} \frac{\ell(Q)^2}{|s-t|^{5/2}}
\,|s - t|^{1/2} \,ds\\
+ 
 \int_{|s-t|>4\ell(Q)^2}\frac1{|s-u|^{3/2}}\,|t-u|^{1/2} \,ds \lesssim 1.
$$

\end{proof}
\vv


\section{Capacities and removable singularities}\label{section-capacities}

Given a bounded set $E\subset \R^{n+1}$,
we define
\begin{equation}\label{eqsupgame1}
\gamma_\Theta(E) =\sup|\langle\nu,1\rangle|,
\end{equation}
where the supremum is taken over all distributions $\nu$ supported on $E$ such that
\begin{equation}\label{eqsupgame2}
\|\nabla_x W*\nu\|_{L^\infty(\R^{n+1})} \leq 1,\qquad \|\partial_t^{1/2} W*\nu\|_{*,p} \leq 1.
\end{equation}
We call $\gamma_\Theta(E)$ the Lipschitz caloric capacity of $E$.
On the other hand, we define the Lipschitz caloric capacity $+$ of $E$, denoted by $\gamma_{\Theta,+}(E)$, 
in the same way as in \rf{eqsupgame1}, but with the supremum restricted to all positive measures $\nu$
supported on $E$ satisfying also \rf{eqsupgame2}. 
Obviously,
$$\gamma_{\Theta,+}(E)\leq \gamma_\Theta(E).$$

Given $\lambda>0$, we consider the parabolic dilation

$$\delta_\lambda(x,t) = (\lambda x,\lambda^{2}t).$$

It is immediate to check that
$$\gamma_\Theta(\delta_\lambda(E))=\lambda^{n+1}\,\gamma_\Theta(E),\qquad \gamma_{\Theta,+}(\delta_\lambda(E))=\lambda^{n+1}\,\gamma_{\Theta,+}(E).$$
\vv

\begin{lemma}
For every Borel set $E\subset\R^{n+1}$,
$$\gamma_{\Theta,+}(E)\leq \gamma_\Theta(E) \lesssim \HH^{n+1}_{\infty,p}(E),$$
and
$$\dim_{H,p}(E)>n+1\quad\Rightarrow \quad \gamma_\Theta(E)>0.$$
\end{lemma}

In the lemma,  $\dim_{H,p}$ stands for the parabolic Hausdorff dimension.

\begin{proof}
The inequality $\gamma_{\Theta,+}(E)\leq \gamma_\Theta(E)$
is trivial, and the arguments for the other statements are standard. Indeed, to prove $\gamma_\Theta(E) \lesssim \HH^{n+1}_{\infty,p}(E)$
 first notice that we can assume $E$ to be compact.
Let $\nu$
be a distribution supported on $E$ such that
\begin{equation}\label{eqadm5}
\|\nabla_x W*\nu\|_{L^\infty(\R^{n+1})} \leq 1,\qquad \|\partial_t^{1/2} W*\nu\|_{*,p} \leq 1,
\end{equation}
and let $\{A_i\}_{i\in I}$ be a collection of sets in $\R^{n+1}$ which cover $E$ and such that
$$\sum_{i\in I} \diam_p(A_i)^{n+1}\leq 2\,\HH^{n+1}_{\infty,p}(E).$$
For each $i\in I$, let $B_i$ be an open parabolic ball centered in $A_i$ with $r(B_i)=\diam_p(A_i)$, so that $E\subset \bigcup_{i\in I} B_i$. By the compactness of $E$ we can assume $I$ to be finite.
By means of a parabolic version of the Harvey-Polking lemma \cite[Lemma 3.1]{Harvey-Polking}, we can construct $C^\infty$ functions $\vphi_i$, $i\in I$, satisfying:
\begin{itemize}
\item $\supp\vphi_i\subset 2B_i$ for each $i \in I$,
\item $\|\nabla_x \vphi_i\|_\infty \lesssim 1/r(B_i)$, $\|\nabla^2_x \vphi_i\|_\infty + 
\|\partial_t \vphi_i\|_\infty \lesssim 1/r(B_i)^2$,
\item $\sum_{i\in I}\vphi_i = 1$ in $\bigcup_{i\in I} B_i$,
\end{itemize}
Hence, by Lemma \ref{lemgrow},
$$|\langle\nu,1\rangle| = \Big|\sum_{i\in I} \langle\nu,\vphi_i\rangle| \lesssim \sum_{i\in I}r(B_i)^{n+1}
= \sum_{i\in I}\diam_p(A_i)^{n+1}\lesssim \HH^{n+1}_{\infty,p}(E).$$
Since this holds for any distribution $\nu$ supported on $E$ satisfying \rf{eqadm5}, we
deduce that $\gamma_\Theta(E) \lesssim \HH^{n+1}_{\infty,p}(E)$.

\vv

To prove the second assertion in the lemma, let $E\subset\R^{n+1}$ be a Borel set satisfying 
$\dim_{H,p}(E)=s>n+1$. We may assume $E$ to be bounded. We may apply a parabolic version of the well known Frostman lemma, which can be proved by arguments analogous to classical ones replacing the usual dyadic lattice in $\R^{n+1}$ by the parabolic lattice $\DD_p$ defined as follows.
For any $k\in\Z$ we consider the family of parabolic cubes $\DD_{p,k}$ of the form
$$\{(x,t)\in\R^{n+1}: i_j 2^{-k}\leq x_j< (i_j+1)2^{-k} \text{ for $1\leq j\leq n$ and }
i_{n+1} 2^{-2k}\leq t< (i_{n+1}+1)2^{-2k}\},$$
where $i_1,\ldots,i_n,i_{n+1}$ are arbitrary integers. Then we let $\DD_p= \bigcup_{k\in\Z}\DD_{p,k}$.

Arguing as in the proof of Frostman lemma in \cite[Theorem 8.8]{Mattila-gmt} or 
\cite[Theorem 1.23]{Tolsa-llibre}, from the fact that $\dim_{H,p}(E)=s>n+1$ 
it follows that there exists some non-zero positive measure $\mu$ supported on $E$ satisfying
$\mu(B_p(\bar x,r))\leq r^s$ for all $\bar x\in\R^{n+1}$ and all $r>0$.
Then, by Lemma \ref{estimates} we deduce that, for all $\bar x\in\R^{n+1}$,
$$|\nabla_x W * \mu(\bar x)|\lesssim \int \frac1{|\bar x-\bar y|_p^{n+1}} \,d\mu(\bar y)\lesssim \diam(E)^{s-(n+1)}
.$$
Now, from Lemma \ref{propo1} it follows that
$$\|\partial_t^{1/2} W * \mu\|_{*,p}<\infty.$$
Therefore, 
$$\gamma_\Theta(E)\geq \frac{\mu(E)}{\max\big(\|\nabla_x W*\mu\|_{L^\infty(\R^{n+1})},\,\|\partial_t^{1/2} W*\mu\|_{*,p} \big)}>0.$$
\end{proof}

\vv
We say that a compact set $E\subset \R^{n+1}$ is Lipschitz removable for the heat equation (or Lipschitz caloric removable) if for any open set $\Omega\subset\R^{n+1}$, any function $f:\R^{n+1}\to\R$
such that
\begin{equation}\label{eqcond22}
\|\nabla_x f\|_{L^\infty(\Omega)}<\infty,\quad \|\partial_t^{1/2} f\|_{*,\Omega,p} <\infty
\end{equation}
satisfying the heat equation in $\Omega\setminus E$, also satisfies the heat equation in the whole $\Omega$.

\begin{rem}
Functions satisfying \rf{eqcond22} are called regular $(1,1/2)$-Lipschitz in the literature (see
\cite{Nystrom-Stromqvist}, for example). So perhaps it would be
more precise to talk about regular $(1,1/2)$-Lipschitz removability or about regular $(1,1/2)$-Lipschitz caloric capacity. However, we have preferred the simpler terminology of Lipschitz removability and Lipschitz
caloric capacity for shortness.
\end{rem}

\vv
\begin{theorem}\label{teoremov}
A compact set $E\subset\R^{n+1}$ is Lipschitz caloric removable if and only if $\gamma_\Theta(E)=0$.
\end{theorem}

\begin{proof}
It is clear that if $E$ is Lipschitz caloric removable, then $\gamma_\Theta(E)=0$. Conversely, 
suppose that $E\subset\R^{n+1}$ is not Lipschitz caloric removable. So there exists some open set
$\Omega\subset \R^{n+1}$ and some function
$f:\R^{n+1}\to\R$ satisfying
$$\|\nabla_x f\|_{L^\infty(\Omega)}<\infty,\qquad \|\partial_t^{1/2} f\|_{*,\Omega,p} <\infty$$
and $\Theta (f)\equiv0$ in $\Omega\setminus E$ but $\Theta (f)\not\equiv0$ in $\Omega$ (in the distributional
sense). So there exists some (open) parabolic cube $Q\subset\Omega$ such that $4Q\subset \Omega$ and $\Theta (f)\not\equiv0$ in $Q$. Let $\chi$ be a non-negative $C^\infty$ function which equals $1$ in $2Q$ and vanishes
in $3Q^c$, and let $\wt f=\chi\,f$. It is immediate to check that $\wt f$ is Lipschitz in $\R^{n+1}$ and 
 $\partial_t^{1/2} \wt f\in BMO_p$. Consider the distribution $\nu=\Theta(\wt f)$. Since $\nu$ does not vanish
identically in $Q$, there exists some $C^\infty$ function $\vphi$ supported on $Q$ such that $\langle \nu,\vphi\rangle>0$. Now take $g=W*(\vphi\nu)$. By Theorem \ref{teoloc},
$$\|\nabla_x g\|_\infty<\infty,\qquad \|\partial_t^{1/2} g\|_{*,p} <\infty,$$
and thus, since $\supp(\vphi\nu)\subset Q\cap E$,
$$\gamma_\Theta(E) \geq \gamma_\Theta(Q\cap E) \geq \frac{\langle \vphi\,\nu,1\rangle}{\max\big(
\|\nabla_x g\|_\infty,\,\|\partial_t^{1/2} g\|_{*,p}\big)} = 
\frac{\langle \nu,\,\vphi\rangle}{\max\big(
\|\nabla_x g\|_\infty,\,\|\partial_t^{1/2} g\|_{*,p}\big)}
>0.$$
\end{proof}
\vv

From the preceding lemmas, it is clear that, for any compact set $E\subset\R^{n+1}$,
\begin{itemize} 
\item if $\dim_{H,p}(E)>n+1$, then $E$ is not Lipschitz caloric removable,
\item if $\HH^{n+1}_{p}(E)=0$ (and so, in particular, if $\dim_{H,p}(E)<n+1$), then 
$E$ is Lipschitz caloric removable.
\end{itemize}
Thus the critical parabolic Hausdorff dimension for Lipschitz caloric removability (and for $\gamma_\Theta$) is $n+1$.

\vv
Next we consider the operator
$$T\nu = \nabla_x W * \nu,$$
defined over distributions $\nu$ in $\R^{n+1}$. When $\mu$ is a finite measure, one can easily check that $T\mu(\bar x)$ is defined for
$m$-a.e.\ $\bar x\in\R^{n+1}$ by the integral 
$$T\mu(\bar x) = \int \nabla_x W(\bar x-\bar y)\,d\mu(\bar y).$$
For $\ve>0$, we also consider the truncated operator
$$T_\ve\mu(\bar x) = \int_{|\bar x-\bar y|>\ve} \nabla_x W(\bar x- \bar y)\,d\mu(\bar y),
$$
whenever the integral makes sense,
and for a function $f\in L^1_{loc}(\mu)$, we write
$$T_{\mu} f\equiv T (f\,\mu),\qquad T_{\mu,\ve} f\equiv T_\ve (f\,\mu).$$
We also denote
$$T_*\mu(x) = \sup_{\ve>0} |T_\ve \mu(x)|,\quad  T_{*,\mu}f(x) = \sup_{\ve>0} |T_\ve(f\, \mu)(x)|.$$
We say that $T_\mu$ is bounded in $L^2(\mu)$ if the operators $T_{\mu,\ve}$ are bounded in $L^2(\mu)$
uniformly on $\ve>0$.

Remark that $T$ is a singular integral operator with Calder\'on-Zygmund kernel in the parabolic space.
More precisely:

\begin{lemma}\label{bound}
The kernel $K\equiv \nabla_xW$ of $T$ satisfies the following:
\begin{itemize}
\item[(a)] $|K(\bar x)|\lesssim \dfrac1{|\bar x|_p^{n+1}}$\; for all $\bar x\neq0$.

\item[(b)] $|\nabla_x K(\bar x)|\lesssim \dfrac1{|\bar x|_p^{n+2}}$ \;and \;
$|\partial_t K(\bar x)|\lesssim \dfrac1{|\bar x|_p^{n+3}}$ for all $\bar x\neq0$.

\item[(c)] For all $\bar x,\bar x'\in \R^{n+1}$ such that $|\bar x- \bar x'|_p\le |\bar x|_p/2$, $\bar x\neq0$,
$$|K(\bar x) - K(\bar x')|\lesssim \frac{|\bar x - \bar x'|_p}{|\bar x|_p^{n+2}}.$$
\end{itemize}
\end{lemma}

\begin{proof}
The estimate in (a) already appears in Lemma \ref{estimates}. The estimates in (b) follow by calculations analogous to the ones in that lemma. Finally, (c) is an easy consequence of (b). Indeed,
given $\bar x,\bar x'\in \R^{n+1}$ such that $|\bar x- \bar x'|_p\le |\bar x|/2$, write
$$\bar x = (x,t),\quad \bar x' = (x',t'),\quad \hat x = (x',t).$$
Then 
\begin{align*}
|K(\bar x) - K(\bar x')| & \leq |K(\bar x) - K(\hat x)| + |K(\hat x) - K(\bar x')|\\
& \leq |x-x'|\,\sup_{y\in[x,x']} |\nabla_x K((y,t))| + 
|t-t'|\,\sup_{s\in[t,t']} |\partial_t K((\bar x',s))| \\
& \lesssim  
\frac{| x - x'|}{|\bar x|_p^{n+2}} + \frac{|t - t'|}{|\bar x|_p^{n+3}} \lesssim \frac{|\bar x -\bar x'|}{|\bar x|_p^{n+2}}.
\end{align*}
\end{proof}

\vv

Recall that given $E\subset\R^{n+1}$, we denote by $\Sigma(E)$ the family of (positive) Borel measures $\mu$ supported on $E$
which have upper parabolic growth of degree $n+1$ with constant $1$, that is,
$$\mu(B_p(\bar x,r))\leq r^{n+1}\quad\mbox{ for all $\bar x\in\R^{n+1},\, r>0$.}$$
Given $E\subset \R^{n+1}$,
we define
\begin{equation}\label{eqsupgame1'}
\wt\gamma_{\Theta,+}(E) =\sup\mu(E),
\end{equation}
where the supremum is taken over all measures $\mu\in\Sigma(E)$ such that
\begin{equation}\label{eqsupgame2'}
\|T\mu\|_{L^\infty(\R^{n+1})} \leq 1,\qquad \|T^*\mu\|_{L^\infty(\R^{n+1})} \leq 1.
\end{equation}
Here $T^*$ is dual of $T$. That is,

$$T^*\mu(\bar x) = \int K(\bar y-\bar x)\,d\mu(\bar y).$$

In the next theorem, among other things, we characterize $\wt\gamma_{\Theta,+}(E)$ in terms of the measures in $\Sigma(E)$
such that $T_\mu$ is bounded in $L^2(\mu)$.

\begin{theorem}
The following holds, for any set $E\subset\R^{n+1}$:
$$
\wt\gamma_{\Theta,+}(E)\lesssim \gamma_{\Theta,+}(E) \approx \sup\{\mu(E):\mu\in\Sigma(E),\,\|T\mu\|_{L^\infty(\R^{n+1})}\leq1\}.
$$
Also,
$$
\wt\gamma_{\Theta,+}(E) \approx\sup\{\mu(E):\mu\in\Sigma(E),\,\|T_\mu\|_{L^2(\mu)\to L^2(\mu)}\leq1\}.
$$
All the implicit constants in the above estimates are independent of $E$.
\end{theorem}

\begin{proof}
Denote
\begin{align*}
S_1& = \sup\{\mu(E):\mu\in\Sigma(E),\,\|T\mu\|_{L^\infty(\R^{n+1})}\leq1\}, \\
S_2 & =\sup\{\mu(E):\mu\in\Sigma(E),\,\|T_\mu\|_{L^2(\mu)\to L^2(\mu)}\leq1\}.
\end{align*}

Notice first the trivial fact that $\wt\gamma_{\Theta,+}(E)\leq S_1$.
The fact that $S_1\gtrsim \gamma_{\Theta,+}(E)$ is an immediate consequence of the definition of 
$\gamma_{\Theta,+}$ and Lemma \ref{lemgrow}. The converse estimate follows from the fact that 
if $\|T\mu\|_{L^\infty(\R^{n+1})}\leq1$ and $\mu$ has upper parabolic growth of degree $n+1$ with constant $1$, then $\|\partial_t^{1/2} W * \mu\|_{*,p}\lesssim 1$, by Lemma \ref{propo1}.

The arguments to show that $\wt\gamma_{\Theta,+}(E)\approx S_2$ are standard. Indeed, let $\mu\in\Sigma(E)$ be such that
$\wt\gamma_{\Theta,+}(E)\leq 2\mu(E)$ and $\|T\mu\|_{L^\infty(\R^{n+1})}\leq1$, $\|T^*\mu\|_{L^\infty(\R^{n+1})}\leq1$. By a Cotlar type inequality analogous to the one in \cite[Lemma 5.4]{Mattila-Paramonov}, say, one deduces that
\begin{equation}\label{ellinf74}
\|T_\ve\mu\|_{L^\infty(\mu)} \lesssim 1,\qquad \|T_\ve^*\mu\|_{L^\infty(\mu)} \lesssim 1,
\end{equation}
uniformly on $\ve>0$.

To obtain the boundedness of the operator $T_{\mu}$ in $L^2(\mu)$ we will use the
  $Tb$ theorem of Hyt\"onen and Martikainen \cite[Theorem 2.3]{Hytonen-Martikainen} for non-doubling measures in geometrically doubling spaces. Remark that the parabolic space is geometrically doubling
  (with the distance $\dist_p$) and thus we can apply that theorem $Tb$ theorem, with the choice $b=1$. 
 Taking into account the conditions \rf{ellinf74}, to ensure that $T_\mu$ is bounded in $L^2(\mu)$,   
  by \cite[Theorem 2.3]{Hytonen-Martikainen} it is enough 
to check that the weak boundedness property 
is satisfied for balls with thin boundaries. That is, for some fixed $A>0$,
\begin{equation}\label{eqweak93}
|\langle T_{\mu,\ve}\chi_B, \chi_B\rangle| \le C\mu(2B),\quad\mbox{ for any a parabolic ball $B\subset\R^{n+1}$ with $A$-thin boundary},
\end{equation}
uniformly on $\ve>0$.
A parabolic ball of radius $r(B)$ is said to have $A$-thin boundary if
\begin{equation}\label{thin}
\mu\{x:\dist_p(x,\partial B)\le tr(B)\}\le A\,t\mu(2B)\quad \mbox{ for all $t\in(0,1)$,}
\end{equation}
See \cite[Lemma 9.43]{Tolsa-llibre} regarding the abundance of such balls, if one chooses $A$ appropriately (just depending on $n$). 

To prove \rf{eqweak93}, let us consider a $C^{\infty}$ function $\vphi$ with compact support in $2B$ such that $\vphi\equiv1$ on $B$ and write
$$
 |\langle T_{\mu,\ve}\chi_B,\chi_B\rangle|\le \int_B|T_{\mu,\ve}\vphi|d\mu+\int_B|T_{\mu,\ve}(\vphi-\chi_B)|d\mu.
$$
Since $\|T\mu\|_{L^\infty(\R^{n+1})}\leq1$, by Lemma \ref{propo1} and Theorem \ref{teoloc} we have $\|T(\vphi\mu)\|_{L^\infty(\R^{n+1})}\leq1$, which in turn implies that $\|T_\ve(\vphi\mu)\|_{L^\infty(\R^{n+1})}\leq1$ uniformly on $\ve>0$. So we deduce that first integral on the right side
is bounded by $C\mu(B)$. 
To get a bound of the second integral we will use that B has a thin boundary and the property (a) in Lemma \ref{bound}. The estimates are very standard but we write the details for
the convenience of the reader:
\begin{align*}
\int_B|T_{\mu,\ve}(\vphi-\chi_B)|d\mu &\lesssim \int_{2B\setminus B}\int_B\frac{d\mu(y)}{|x-y|_p^{n+1}}d\mu(x)\\
&
\le \sum_{j\ge 0}\int_{\{x\notin B:\dist_p(x;\partial B)\approx\frac{r(B)}{2^j}\}}\int_B\frac{d\mu(y)}{|x-y|_p^{n+1}}d\mu(x).
\end{align*}
Given $j$  and $x\notin B$ such that $\dist_p(x,\partial B)\approx\frac{r(B)}{2^j},$ since $\mu\in \Sigma(E)$ one has 
$$
\int_B\frac{d\mu(y)}{|x-y|_p^{n+1}}d\mu(x)\lesssim\sum_{k=-1}^{k=j}\int_{|x-y|_p\approx\frac{r(B)}{2^k}} 
\frac{d\mu(y)}{|x-y|_p^{n+1}}d\mu(x)\lesssim \sum_{k=-1}^{k=j}\frac{\mu(B(x,2^{-k}r(B))}{(r(B)2^{-k})^{n+1}}\lesssim j+2. 
$$
Therefore, by \rf{thin}
$$
\int_B|T_{\mu,\ve}(\vphi-\chi_B)|\,d\mu\lesssim 
\sum_{j\ge 1}(j+2)\,\mu(\{x: \dist_p(x,\partial B)\approx 2^{-j}r(B)\})\lesssim \sum_{j\ge 0}\frac{j+2}{2^j}\mu(2B) \lesssim \mu(2B).
$$
Consequently,  the weak boundedness property \rf{eqweak93} holds and so  
 $T_\mu$ is bounded in $L^2(\mu)$, with
$\|T_\mu\|_{L^2(\mu)\to L^2(\mu)}\lesssim1$. This gives that $S_2\gtrsim \wt\gamma_{\Theta,+}(E)$.

\vv
To prove the converse estimate, let $\mu\in\Sigma(E)$ be such that $\|T_\mu\|_{L^2(\mu)\to L^2(\mu)}\leq1$
and $S_2\leq 2\mu(E)$. From the $L^2(\mu)$ boundedness of $T_\mu$, one deduces that $T$ and $T^*$ are bounded
from the space of finite signed measures $M(\R^{n+1})$ to $L^{1,\infty}(\mu)$. That is, there exists some constant $C>0$ such that for any measure
$\nu\in M(\R^{n+1})$, any $\ve>0$, and any $\lambda>0$,
$$\mu\big(\big\{x\in\R^{n+1}:|T_\ve \nu(x)|>\lambda \big\}\big) \leq C\,\frac{\|\nu\|}\lambda,$$
and the same replacing $T_\ve$ by $T^*_\ve$.
The proof of this fact is analogous to the one of Theorem 2.16 in \cite{Tolsa-llibre}\footnote{For the application of the arguments in \cite{Tolsa-llibre}, notice that the Besicovitch
covering theorem with respect to parabolic balls is valid. Alternatively, see Theorem 5.1 from
\cite{NTV-weak11}.}.
Then, by a well known dualization of these estimates (essentially due to Davie and {\O}ksendal) and an application of Cotlar's inequality, one deduces
that there exists some function $h:E\to [0,1]$ such that
$$\mu(E)\leq C\,\int h\,d\mu,\quad \; \|T(h\,\mu)\|_{L^\infty(\R^{n+1})}\leq 1,\quad \; \|T^*(h\,\mu)\|_{L^\infty(\R^{n+1})}\leq 1.$$
See Theorem 4.6 and Lemma 4.7 from \cite{Tolsa-llibre} for the analogous arguments in the case of analytic capacity and also Lemma 4.2 from \cite{Mattila-Paramonov} for the precise vectorial version of the dualization of the weak $(1,1)$ estimates required in our situation, for example. So we have
$$\wt\gamma_{\Theta,+}(E)\geq \int h\,d\mu\approx \mu(E)\approx S_2.$$
\end{proof}

\vv
\begin{example}
From the preceding theorem we deduce that any subset of positive measure $\HH^{n+1}_p$ of a
regular Lip$(1,1/2)$ graph is non-removable. In particular, any subset of positive measure $\HH^{n+1}_p$ of 
a non-horizontal hyperplane (i.e., not parallel to $\R^n\times\{0\}$) is non-removable.

Remark that any horizontal plane has parabolic Hausdorff dimension $n$, and thus any subset of a horizontal plane is removable.
\end{example}

\vv


\section{The existence of removable sets with positive measure $\HH^{n+1}_p$}

We need the following result, which is of independent interest.

\begin{theorem}\label{teomeas}
Let $E\subset\R^{n+1}$ be a compact set such that $\HH^{n+1}_p(E)<\infty$.
Let $\nu$ be a distribution supported on $E$ such that 
$$\|\nabla_x W*\nu\|_\infty \leq 1,\qquad \|\partial_t^{1/2} W*\nu\|_{*,p} \leq 1.$$
Then $\nu$ is a signed measure which is absolutely continuous with respect to $\HH^{n+1}_p|_E$ and 
there exists a Borel function $f:E\to\R$ such that $\nu=f\,\HH^{n+1}_p|_E$ satisfying 
$\|f\|_{L^\infty(\HH^{n+1}_p|_E)}\lesssim 1$.
\end{theorem}

This theorem is an immediate consequence of Lemma \ref{lemgrow} and the following result.

\begin{lemma}
Let $E\subset\R^{n+1}$ be a compact set such that $\HH^{n+1}_p(E)<\infty$.
Let $\nu$ be a distribution supported on $E$  which has upper parabolic $1$ growth of degree $n+1$.
Then $\nu$ is a signed measure which is absolutely continuous with respect to $\HH^{n+1}_p|_E$ and 
there exists a Borel function $f:E\to\R$ such that $\nu=f\,\HH^{n+1}_p|_E$ satisfying 
$\|f\|_{L^\infty(\HH^{n+1}_p|_E)}\lesssim 1$.
\end{lemma}

\begin{proof}
First we will show that $\nu$ is a signed measure. By the Riesz representation theorem, it is enough
to show that, for any $C^\infty$ function $\psi$ with compact support,
\begin{equation}\label{eq001}
|\langle \nu,\psi\rangle|\leq C(E)\,\|\psi\|_\infty,
\end{equation}
where $C(E)$ is some constant depending on $E$.

To prove \rf{eq001}, we fix $\ve>0$ and we consider a family of open parabolic cubes $Q_i$, 
$i\in I_\ve$, such that
\begin{itemize}
\item $E\subset \bigcup_{i\in I_\ve} Q_i$,
\item $\ell(Q_i)\leq \ve$ for all $i\in I_\ve$, and
\item $\sum_{i\in I_\ve} \ell(Q_i)^{n+1} \leq C\,\HH_p^{n+1}(E) + \ve$.
\end{itemize}
Since $E$ is compact, we can assume that $I_\ve$ is finite. 
By standard arguments, we can find a family of non-negative functions $\vphi_i$, $i\in I_\ve$, such that
\begin{itemize}
\item each $\vphi_i$ is supported on $2Q_i$ and $c\vphi_i$ is admissible for $2Q_i$, for some
absolute constant $c>0$, 
\item $\sum_{i\in I_\ve} \vphi_i = 1$ on $\bigcup_{i\in I_\ve} Q_i$, and in particular on $E$.
\end{itemize}
Indeed, to construct the family of functions $\vphi_i$ we can cover each cube $Q_i$ by a bounded number (depending on $n$)  
dyadic parabolic cubes $R_i^1,\cdots,R_i^m$ with side length $\ell(R_i^j)\leq \ell(Q_i)/8$ and then apply
the usual Harvey-Polking lemma (\cite[Lemma 3.1]{Harvey-Polking} to the family of cubes $\{R_{i,j}\}$.

We write
$$|\langle \nu,\psi\rangle|\leq \sum_{i\in I_\ve} |\langle \nu,\vphi_i\psi\rangle|.$$
For each $i\in I_\ve$, consider the function
$$\eta_i= \frac{\vphi_i\,\psi}{\|\psi\|_\infty + \ell(Q_i)\,\|\nabla_x \psi\|_\infty +
\ell(Q_i)^2\,\|\partial_t \psi\|_\infty + \ell(Q_i)^2\,\|\Delta \psi\|_\infty}.$$
We claim that $c\,\eta_i$ is admissible for $2Q_i$, for some absolute constant $c>0$. To check this,
just note that $\vphi_i\,\psi$ is supported on $2Q_i$ and satisfies
$$\|\nabla_x (\vphi_i\,\psi)\|_\infty \leq \|\nabla_x \vphi_i\|_\infty\,\|\psi\|_\infty
+ \| \vphi_i\|_\infty\,\|\nabla_x\psi\|_\infty \lesssim \frac1{\ell(Q_i)}\,\|\psi\|_\infty
+ \|\nabla_x\psi\|_\infty.$$
Hence, 
$$\|\nabla_x \eta_i\|_\infty \lesssim \frac1{\ell(Q_i)}.$$
Analogously,
$$\|\partial_t (\vphi_i\,\psi)\|_\infty \leq \|\partial_t \vphi_i\|_\infty\,\|\psi\|_\infty
+ \| \vphi_i\|_\infty\,\|\partial_t\psi\|_\infty \lesssim \frac1{\ell(Q_i)^2}\,\|\psi\|_\infty
+ \|\partial_t\psi\|_\infty,$$
and so
$$\|\partial_t \eta_i\|_\infty \lesssim \frac1{\ell(Q_i)^2}.$$
Also,
\begin{align*}
\|\Delta (\vphi_i\,\psi)\|_\infty & \leq \|\Delta \vphi_i\|_\infty\,\|\psi\|_\infty
 + 2 \,\|\nabla_x \vphi_i\|_\infty\,\|\nabla_x\psi\|_\infty + \| \vphi_i\|_\infty\,\|\Delta\psi\|_\infty\\
&\lesssim \frac1{\ell(Q_i)^2}\,\|\psi\|_\infty + \frac1{\ell(Q_i)}\,\|\nabla_x\psi\|_\infty 
+ \|\Delta\psi\|_\infty,
\end{align*}
and thus
$$\|\Delta \eta_i\|_\infty \lesssim \frac1{\ell(Q_i)^2}.$$
So the claim above holds and, consequently, by the assumptions in the lemma,
$$|\langle \nu,\eta_i\rangle|\lesssim \ell(Q_i)^{n+1}.$$

From the preceding estimate, we deduce that
$$|\langle \nu,\psi\rangle|\leq \sum_{i\in I_\ve} |\langle \nu,\vphi_i\psi\rangle|
\lesssim \sum_{i\in I_\ve} \ell(Q_i)^{n+1} \,\big(
\|\psi\|_\infty + \ell(Q_i)\,\|\nabla_x \psi\|_\infty +
\ell(Q_i)^2\,\|\partial_t \psi\|_\infty + \ell(Q_i)^2\,\|\Delta \psi\|_\infty\big).$$
Since $\ell(Q_i)\leq\ve$ for each $i$, we infer that
\begin{align*}
|\langle \nu,\psi\rangle| &
\lesssim \sum_{i\in I_\ve} \ell(Q_i)^{n+1} \,\big(
\|\psi\|_\infty + \ve\,\|\nabla_x \psi\|_\infty +
\ve^2\,\|\partial_t \psi\|_\infty + \ve^2\,\|\Delta \psi\|_\infty\big)\\
&\lesssim \big(\HH^{n+1}_p(E) + \ve\big)\,
\big(
\|\psi\|_\infty + \ve\,\|\nabla_x \psi\|_\infty +
\ve^2\,\|\partial_t \psi\|_\infty + \ve^2\,\|\Delta \psi\|_\infty\big).
\end{align*}
Letting $\ve\to0$, we get
$$|\langle \nu,\psi\rangle|\lesssim \HH^{n+1}_p(E)\,\|\psi\|_\infty,$$
which gives \rf{eq001} and proves that $\nu$ is a finite signed measure, as wished.

It remains to show that
there exists some Borel function $f:E\to\R$ such that $\nu=f\,\HH^{n+1}_p|_E$, with 
$\|f\|_{L^\infty(\HH^{n+1}_p|_E)}\lesssim1$.
To this end, let $g$ be the density of $\nu$ with respect to its variation $|\nu|$, so that
$\nu = g\,|\nu|$ with $g(\bar x)=\pm1$ for $|\nu|$-a.e.\ $\bar x\in\R^{n+1}$.
We will show that
\begin{equation}\label{eqdens92}
\limsup_{r\to0}\frac{|\nu|(B_p(\bar x,r))}{r^{n+1}}\lesssim1 \quad\mbox{ for $|\nu|$-a.e.\ $\bar x\in\R^{n+1}$}.
\end{equation}
This implies that $|\nu|=\wt f\,\HH^{n+1}_p|_E$ for some non-negative function $\wt f\lesssim1$. This fact is
well known  if one replaces parabolic balls by Euclidean balls and parabolic Hausdorff measure
by the usual Hausdorff measure (see Theorem 6.9 from \cite{Mattila-gmt}). The arguments extend easily 
to the parabolic case thanks to the validity of the Besicovitch covering theorem with respect to  parabolic balls.

So to complete the proof of the lemma it suffices to show \rf{eqdens92} (since then we will have
$\nu= g\,\wt f\,\HH^{n+1}_p|_E$ with $|g\wt f|\lesssim1$). Notice that, by the Lebesgue differentiation theorem, 
$$\lim_{r\to 0}\frac1{|\nu|(B_p(\bar x,r))} \int_{B_p(\bar x,r)}|g(\bar y)-g(\bar x)|\,d|\nu|(\bar y) =0
\quad\mbox{ for $|\nu|$-a.e.\ $\bar x\in\R^{n+1}$}$$
(because of the validity of the Besicovitch covering theorem with respect to the parabolic balls again).
Let $\bar x\in E$ be a Lebesgue point for $|\nu|$ with $|g(\bar x)|=1$,  let $\ve>0$ to be chosen below,
 and let $r_0>0$ be small enough such that, for $0<r\leq r_0$,
$$\frac1{|\nu|(B_p(\bar x,r))} \int_{B_p(\bar x,r)}|g(\bar y)-g(\bar x)|\,d|\nu|(\bar y)<\ve.$$
Suppose first that
\begin{equation}\label{eqdoub*}
|\nu|(B_p(\bar x,2r))\leq 2^{n+3} \,|\nu|(B_p(\bar x,r)),
\end{equation}
and let $\vphi_{\bar x,r}$ be some non-negative $C^\infty$ function supported on $B_p(\bar x,2r)$ which equals $1$ on 
$B_p(\bar x,r)$ such that $c\,\vphi_{\bar x,r}$ is admissible for the smallest parabolic cube $Q$ containing $B_p(\bar x,2r)$,
so that
$$\left| \int \vphi_{\bar x,r}\,d\nu\right|\lesssim r^{n+1}.$$
Now observe that
\begin{align*}
\left|\int \vphi_{\bar x,r}\,d\nu - g(\bar x)\!\int \vphi_{\bar x,r}\,d|\nu| \right|
& = \left|\int \vphi_{\bar x,r}(\bar y) (g(\bar y)  - g(\bar x))\,d|\nu|(\bar y) \right| \lesssim \int_{B_p(\bar x,2r)} |g(\bar y)  - g(\bar x)|\,d|\nu|(\bar y) \\
& \leq \ve\,|\nu|(B_p(\bar x,2r))\leq \ve\,2^{n+3}\,|\nu|(B_p(\bar x,r)) \lesssim \ve \int \vphi_{\bar x,r}\,d\nu.
\end{align*}
Thus, if $\ve$ is chosen small enough, we deduce that
$$ \int \vphi_{\bar x,r}\,d|\nu| = |g(\bar x)|\int \vphi_{\bar x,r}\,d|\nu| \leq 2
\left|\int \vphi_{\bar x,r}\,d\nu\right|\lesssim r^{n+1}.$$
Therefore, using again that $\vphi_{\bar x,r}=1$ on $B_p(\bar x,r)$, we get
\begin{equation}\label{eqgro956}
|\nu|(B_p(\bar x,r))\lesssim r^{n+1}.
\end{equation}

To get rid of the doubling assumption \rf{eqdoub*}, notice that for $|\nu|$-a.e.\ $\bar x\in\R^{n+1}$ there
exists a sequence of balls $B_p(\bar x,r_k)$, with $r_k\to0$, satisfying \rf{eqdoub*} (we say that the balls
$B_p(\bar x,r_k)$ are $|\nu|$-doubling). Further, we may assume
that $r_k=2^{h_k}$, for some $h_k\in\N$. The proof of this fact is analogous to the one of Lemma 2.8
in \cite{Tolsa-llibre}. So for such a point $\bar x$, by the arguments above, we know that there exists
some $k_0>0$ such that
$$|\nu|(B_p(\bar x,r_k))\lesssim r_k^{n+1}\quad\mbox{ for $k\geq k_0$,}$$
assuming also that $\bar x$ is a $|\nu|$-Lebesgue point for the density $g$. Given an arbitrary $r\in (0,r_{k_0})$,
let $j$ be the smallest integer $r\leq 2^j$, and let $2^k$ be the smallest $j\geq k$ such that
the ball $B_p(\bar x,2^k)$ is $|\nu|$-doubling (i.e., \rf{eqdoub*} holds for this ball). Observe that
$2^k\leq r_{k_0}$.
Then, taking into account that the balls $B_p(\bar x,2^h)$ are non-doubling for $k\leq h<j$ and applying \rf{eqgro956} for
$r=2^k$, we obtain
$$|\nu|(B_p(\bar x,r))\leq |\nu|(B_p(\bar x,2^j)) \leq 2^{(n+3)(j-k)}\,|\nu|(B_p(\bar x,2^k)) \lesssim 2^{(n+3)(j-k)}\,2^{k(n+1)}
\leq 2^{j(n+1)}\approx r^{n+1}.$$
Hence, \rf{eqdens92} holds and we are done.
\end{proof}

\vvv

Next we will construct a self-similar Cantor set $E\subset \R^3$ with positive and finite measure $\HH^3_p$ and we will show that it is removable. For simplicity we work in $\R^3$, although this construction extends easily
to $\R^{n+1}$, with $n\geq 1$ arbitrary. Our example is inspired by the typical planar $1/4$ Cantor
set in the setting of analytic capacity (see \cite{Garnett} or \cite[p.\ 35]{Tolsa-llibre}, for example).

We construct the Cantor set $E$ as follows:
We let $E_0=Q^0=[0,1]^3$ (i.e., $Q_0$ is the unit cube). Next we replace $Q^0$ by $12$ disjoint closed parabolic cubes $Q^1_i$ with side length $12^{-1/3}$ located in the following positions: they are all contained in $Q^0$ and eight of them contain each one a vertex of $Q^0$. The centers of the remaining other four cubes $Q^1_i$ are in
the plane $\{(x_1,x_2,t):\textcolor{red}{t=1/2}\}$ and each one of these cubes has one of its vertical edges contained in one of the vertical edges of $Q^0$. In this way, the vertical projection of the set $E_1=\bigcup_{i=1}^{12}Q_i^1$
consists of $4$ squares, and the two horizontal projections parallel to the horizontal axes consist of $6$ Euclidean rectangles each one. 

We proceed inductively:
In each generation $k$, we replace each parabolic cube $Q_j^{k-1}$ of the previous generation by $12$ parabolic cubes $Q_i^{k}$ with side length $12^{-k/3}$ which are contained in $Q_j^{k-1}$ and located in the same relative position to $Q_j^{k-1}$ as the cubes $Q_1^1,\ldots,Q^1_{12}$ with respect to $Q_0$. 

Notice that in each generation $k$ there are $12^k$ closed parabolic cubes with side length $12^{-k/3}$. 
We denote by $E_k$ the union of all these parabolic cubes from the $k$-th generation. By construction, $E_k\subset E_{k-1}$.
We let 
\begin{equation}\label{defE}
E=\bigcap_{k=0}^\infty E_k.
\end{equation}
It is easy to check that 
$\dist_p(Q_i^k,Q_h^k)\gtrsim 12^{-k/3}$ for $i\neq h$, and if $Q_i^k$, $Q_h^k$ are contained in the same parabolic cube $Q^{k-1}_j$, then $\dist_p(Q_i^k,Q_h^k)\lesssim 12^{-k/3}$.
Taking into account that, for each $k\geq0$,
$$\sum_{i=0}^{12^k} \ell(Q_i^k)^3 = 12^k\cdot (12^{-k/3})^3=1,$$
by standard arguments it follows that
$$0<\HH^3_p(E)<\infty.$$
Further, $\HH^3_p|_E$ coincides, modulo a constant factor, with the probability measure $\mu$ supported on $E$ which gives the same measure to all the cubes $Q_i^k$ of the same generation $k$ (i.e., $\mu(Q_i^k)=12^{-k}$).

\vv

\begin{theorem}
The Cantor set $E$ defined in \rf{defE} is Lipschitz caloric removable.
\end{theorem}

\begin{proof}
We will suppose that $E$ is not removable and we will reach a contradiction. 
By Theorem \ref{teoremov}, there exists a distribution $\nu$ supported on $E$ such that 
$|\langle\nu,1\rangle|>0$ and
$$
\|\nabla_x W*\nu\|_{L^\infty(\R^{n+1})} \leq 1,\qquad \|\partial_t^{1/2} W*\nu\|_{*,p} \leq 1.
$$
By Theorem \ref{teomeas}, $\nu$ is a signed measure of the form
$$\nu=f\,\mu,\qquad \|f\|_{L^\infty(\mu)}\lesssim 1,$$
where $\mu$ is the probability measure supported on $E$ such that
$\mu(Q_i^k)=12^{-k}$ for all $i,k$.
It is easy to check that $\mu$ (and thus $|\nu|$) has upper parabolic growth of degree $3$. Then, 
arguing as in \cite[Lemma 5.4]{Mattila-Paramonov}, it follows that there exists some constant $K$ such that
\begin{equation}\label{eqmatpar}
T_*\nu(\bar x)\leq K\,\quad \mbox{ for all $\bar x\in\R^{n+1}$.}
\end{equation}

For $\bar x\in E$, we denote by $Q_{\bar x}^k$ the cube $Q_i^k$ that contains $\bar x$. Then we consider the auxiliary operator
$$\wt T_*\nu(\bar x) = \sup_{k\geq0} |T (\chi_{\R^3\setminus Q_{\bar x}^k}\nu)(\bar x)|.$$
By the separation condition between the cubes $Q_i^k$, the upper parabolic growth of $|\nu|$, and the condition \rf{eqmatpar}, it follows easily that
\begin{equation}\label{eqmatpar'}
\wt T_*\nu(\bar x)\leq K'\,\quad \mbox{ for all $\bar x\in E$,}
\end{equation}
for some fixed constant $K'$.

We will contradict the last estimate. To this end, consider a Lebesgue point $\bar x_0\in E$ (with respect to $\mu$ and to parabolic cubes) of the density $f=\frac{d\nu}{d\mu}$ such that $f(\bar x_0)>0$. The existence of this point is
guarantied by the fact that $\nu(E)>0$.
Given $\ve>0$ small enough to be chosen below, consider a parabolic cube $Q_i^k$ containing $\bar x_0$ such that
$$\frac1{\mu(Q_i^k)} \int_{Q_i^k} |f(\bar y)- f(\bar x_0)|\,d\mu(y)\leq \ve.$$
Given $m\gg1$, to be fixed below too, it is easy to check that if $\ve$ is chosen small enough (depending on $m$ and on $f(\bar x_0)$), then
the above condition ensures that
every cube $Q_j^h$ contained in $Q_i^k$ such that $k\leq h\leq k+m$ satisfies
\begin{equation}\label{eqmatp*}
\frac12 f(\bar x_0)\,\mu(Q^h_j)\leq
\nu(Q^h_j)\leq 2 f(\bar x_0)\,\mu(Q^h_j).
\end{equation}
Notice also that, writing $\nu=\nu^+-\nu^-$, since $f(\bar x_0)>0$,
$$\nu^-(Q_i^k) = \int_{Q_i^k} f^-(\bar y)\,d\mu(\bar y) \leq \int_{Q_i^k} |f(\bar y)-f(\bar x_0)|\,d\mu(\bar y) \leq \ve\,\mu(Q_i^k).$$

Let $\bar z= (z_1,z_2,,u)$ be one of the two upper leftmost corners of $Q^k_i$ (i.e., with $z_1$ minimal 
and $u$ maximal in $Q^k_i$). Since
$|T(\chi_{Q^k_{\bar z}\setminus Q^{k+m}_{\bar z}}\nu(\bar z)| \leq 2\,\wt T_*\nu(\bar z)$, we have
$$\wt T_*\nu(\bar z) \geq \frac12\,|T(\chi_{Q^k_{\bar z}\setminus Q^{k+m}_{\bar z}}\nu)(\bar z)| \geq 
\frac12\,|T(\chi_{Q^k_{\bar z}\setminus Q^{k+m}_{\bar z}}\nu^+)(\bar z)| - \frac12\,|T(\chi_{Q^k_{\bar z}\setminus Q^{k+m}_{\bar z}}\nu^-)(\bar z)|.
$$
Using the fact that $\dist_p(z,\,Q^k_{\bar z}\setminus Q^{k+m}_{\bar z}) \gtrsim \ell(Q^{k+m}_{\bar z})$, we get
$$|T(\chi_{Q^k_{\bar z}\setminus Q^{k+m}_{\bar z}}\nu^-)(\bar z)|\lesssim \frac{\nu^-(Q^k_i)}{\ell(Q^{k+m}_{\bar z})^3} 
\leq \ve\,\frac{\mu(Q^k_i)}{\ell(Q^{k+m}_{\bar z})^3} = \ve\,\frac{\mu(Q^k_i)}{12^{-m}\ell(Q^{k}_i)^3} 
\lesssim 12^{m}\,\ve.$$

To estimate $|T(\chi_{Q^k_{\bar z}\setminus Q^{k+m}_{\bar z}}\nu^+)(\bar z)|$ from below, recall that the first component of the
kernel $K=\nabla_x W$ equals
$$K_1(\bar x) = c_0\,\frac{-x_1}{t^{2}}\,e^{-|x|^2/(4t)}\,\chi_{\{t>0\}},$$
for some absolute constant $c_0>0$.
Then, by the choice of $\bar z$, it follows that
\begin{equation}\label{eqpos1}
K_1(\bar z- \bar y) \geq 0\quad \mbox{ for all $\bar y\in Q^k_{\bar z}\setminus Q^{k+m}_{\bar z}$.}
\end{equation}
We write
$$|T(\chi_{Q^k_{\bar z}\setminus Q^{k+m}_{\bar z}}\nu^+)(\bar z)| \geq \int_{Q^k_{\bar z}\setminus Q^{k+m}_{\bar z}} K_1(\bar z- \bar y)\,
d\nu^+(\bar y) = \sum_{h=k}^{k+m-1} \int_{Q^h_{\bar z}\setminus Q^{h+1}_{\bar z}} K_1(\bar z- \bar y)\,
d\nu^+(\bar y).$$
Taking into account \rf{eqpos1} and the fact that, for $k\leq h\leq k+m-1$, $Q^h_{\bar z}\setminus Q^{h+1}_{\bar z}$ contains a cube $Q^{h+1}_j$ such that for all $\bar y=(y_1,y_2,s)$,
$$0<y_1-z_1 \approx |\bar y-\bar z|\approx\ell(Q^{h+1}_j),\qquad 0<u-s \approx \ell(Q^{h+1}_j)^2,$$
using also \rf{eqmatp*},
we deduce
$$\int_{Q^h_{\bar z}\setminus Q^{h+1}_{\bar z}} K_1(\bar z- \bar y)\,
d\nu^+(\bar y) \gtrsim \frac{\nu^+(Q^{h+1}_j)}{\ell(Q^{h+1}_j)^3}\gtrsim f(\bar x_0)\,\frac{\mu(Q^{h+1}_j)}{\ell(Q^{h+1}_j)^3}
= f(\bar x_0),$$
Thus,
$$|T(\chi_{Q^k_{\bar z}\setminus Q^{k+m}_{\bar z}}\nu^+)(\bar z)| \gtrsim (m-1)\,f(\bar x_0).$$
Together with the previous estimate for $|T(\chi_{Q^k_{\bar z}\setminus Q^{k+m}_{\bar z}}\nu^-)(\bar z)|$, this tells us that
$$\wt T_*\nu(\bar z) \gtrsim (m-1)\,f(x_0) - C\,12^{m}\,\ve,$$
for some fixed $C>0$.
It is clear that if we choose $m$ big enough and then $\ve$ small enough, depending on $m$, this
lower bound contradicts \rf{eqmatpar'}, as wished.
\end{proof}


\vvv


\begin{thebibliography}{NTWV}

\bibitem[Ah]{Ahlfors} 
L. V. Ahlfors. {\em Bounded analytic functions.} Duke Math. J., 14:1-11, 1947. 

\bibitem[Da]{David-vitus} G. David, {\em Unrectifiable $1$-sets have vanishing analytic capacity,} Revista Mat. Iberoamericana 14(2) (1998), 369--479.

\bibitem[Ga]{Garnett} J. Garnett. {\em Positive length but zero analytic capacity.} Proc. Amer. Math. Soc. 24 (1970), 696-699.

\bibitem[HP]{Harvey-Polking} R. Harvey and J. Polking. {\em Removable singularities of solutions of linear partial differential equations.} Acta Math. 125 (1970), 39--56.

\bibitem[Ho1]{Hofmann} S. Hofmann. {\em A characterization of commutators of parabolic singular integrals.} Fourier Analysis and partial differential equations (Miraflores de la Sierra, 1992), 195-210, Stud. Adv. Math., CRC, Boca Raton, FL,1995.

\bibitem[Ho2]{Hofmann-Duke} S. Hofmann. {\em Parabolic singular integrals of Calder\'on type, rough operators, and
caloric layer potentials.} Duke Math. J. 90 (1997), no. 2, 209--259.

\bibitem[HoL1]{Hofmann-Lewis1} S. Hofmann and J. Lewis. {$L^2$ solvability and representation by caloric layer
potentials in time-varying domains.} Ann. of Math. (2) 144 (1996), no. 2, 349--420.

\bibitem[HoL2]{Hofmann-Lewis2} S. Hofmann and J. Lewis. {\em The Dirichlet problem for parabolic operators with
singular drift terms.} Mem. Amer. Math. Soc. 151 (2001), no. 719, viii+113 pp.

\bibitem[HoLN1]{HLN1}  S. Hofmann, J. Lewis and K.Nystr\"om. {\em Existence of big pieces of graphs for
parabolic problems.} Ann. Acad. Sci. Fenn. Math. 28 (2003), no. 2, 355-384.

\bibitem[HoLN2]{HLN2} S. Hofmann, J. Lewis and K. Nystr\"om. {\em Caloric measure in parabolic flat
domains.} Duke Math. J. 122 (2004), no. 2, 281-346.

\bibitem[HyM]{Hytonen-Martikainen} T. Hyt\"onen, and H. Martikainen. {\em Non-homogeneous $Tb$ theorem and random dyadic cubes on metric measure spaces.} J. Geom. Anal. 22 (2012), no. 4, 1071--1107.

\bibitem[LM]{Lewis-Murray} J. Lewis and M. Murray. {\em The method of layer potentials for the heat equation
in time-varying domains.} Mem. Amer. Math. Soc. 114 (1995), no. 545, viii+157 pp.

\bibitem[LS]{Lewis-Silver} J. Lewis and J. Silver. {\em Parabolic measure and the Dirichlet problem for the
heat equation in two dimensions.} Indiana Univ. Math. J. 37 (1988), no. 4, 801-839.

\bibitem[Ma]{Mattila-gmt} P. Mattila. {\em Geometry of sets and measures in Euclidean spaces.} Cambridge Studies in Advanced Mathematics, vol. 44, Cambridge University Press, Cambridge, 1995.

\bibitem[MaP]{Mattila-Paramonov} P. Mattila and P.V. Paramonov. {\em On geometric properties of harmonic Lip1-capacity.} Pacific J. Math. 171 (1995), no. 2, 469--491. 

\bibitem[NToV1]{NToV1} F. Nazarov, X. Tolsa , and A. Volberg. {\em 
On the uniform rectifiability of AD-regular measures with bounded Riesz transform operator: the case of codimension 1.} Acta Math. 213:2 (2014), 237--321. 

\bibitem[NToV2]{NToV2} F. Nazarov, X. Tolsa , and A. Volberg. {\em
The Riesz transform, rectifiability, and removability for Lipschitz harmonic functions}. Publ. Mat. 58:2 (2014), 517--532. 

\bibitem[NTV]{NTV-weak11}
F. Nazarov, S. Treil, and A. Volberg. {\em Weak type estimates and Cotlar inequalities for Calder\'on-Zygmund operators on nonhomogeneous spaces.} Int. Math. Res. Not. 9, 463--487 (1998).

\bibitem[NySt]{Nystrom-Stromqvist} K. Nystr\"om and  M. Str\"omqvist.
{\em On the parabolic Lipschitz approximation of parabolic uniform rectifiable sets.}
Rev. Mat. Iberoam. 33 (2017), no. 4, 1397--1422. 

\bibitem[P]{paramonov} P. Paramonov. {\em On harmonic approximation in the $\mathcal C^1-$norm}, Math. USSR Sbornik, Vol. 71 (1992), no. 1, 183-207.

\bibitem[Pr]{tams} L. Prat. {\em On the semiadditivity of the capacities associated with signed vector valued Riesz kernels.} Trans. Amer. Math. Soc., 364(11):5673-5691, 2012.

\bibitem[To1]{tolsasemiad} X. Tolsa. {\em Painleve's problem and the semiadditivity of analytic capacity.} Acta Math., 190(1):105-149, 2003.


\bibitem[To2]{Tolsa-llibre}
X.~Tolsa.
{\em Analytic capacity, the {C}auchy transform, and non-homogeneous
  {C}alder\'on-{Z}ygmund theory}, volume 307 of { Progress in Mathematics}.
 Birkh\"auser Verlag, Basel, 2014.

\bibitem[Vo]{volbergsemiad} A. Volberg. {\em  Calder\'on-Zygmund capacities and operators on nonhomogeneous spaces.} Volume 100 of CBMS Regional Conference Series in Mathematics.
Published for the Conference Board of the Mathematical Sciences, Washington,
DC, 2003.
\end{thebibliography}
\end{document}